\documentclass[12pt, a4paper]{amsart}
\usepackage{amsmath}
\usepackage{geometry,amsthm,graphics,tabularx,amssymb,shapepar}
\usepackage{amscd}
\usepackage[all]{xypic}

\newcommand{\CF}{{\mathcal {F}}}

\newcommand{\CK}{{\mathcal {K}}}

\newcommand{\CR}{{\mathcal {R}}}
\newcommand{\CS}{{\mathcal {S}}}
\newcommand{\CT}{{\mathcal {T}}}

\newcommand{\CW}{{\mathcal {W}}}

\newcommand{\RA}{{\mathrm {A}}}

\newcommand{\RE}{{\mathrm {E}}}
\newcommand{\RF}{{\mathrm {F}}}

\newcommand{\RK}{{\mathrm {K}}}

\newcommand{\RN}{{\mathrm {N}}}

\newcommand{\RQ}{{\mathrm {Q}}}

\newcommand{\disc}{{\mathrm{disc}}}

\newcommand{\GL}{{\mathrm{GL}}}

\newcommand{\Hom}{{\mathrm{Hom}}}

\newcommand{\rank}{{\mathrm{rank}}}

\newcommand{\Sp}{{\mathrm{Sp}}}

\newcommand{\cover}[1]{\widetilde{#1}}

\newcommand{\vsp}{{\vspace{0.2in}}}
\newcommand{\vvsp}{{\vspace{0.1in}}}

\newcommand{\Irr}{\operatorname{Irr}}

\newcommand{\con}{\textit{C}}

\newcommand{\oH}{\operatorname{H}}
\newcommand{\oO}{\operatorname{O}}

\newcommand{\oU}{\operatorname{U}}
\newcommand{\oG}{\operatorname{G}}

\newcommand{\oP}{\operatorname{P}}
\newcommand{\oJ}{\operatorname{J}}

\newcommand{\oGb}{\operatorname{\bar{G}}}
\newcommand{\oJb}{\operatorname{\bar{J}}}
\newcommand{\oPb}{\operatorname{\bar{P}}}
\newcommand{\oHil}{\operatorname{Hil}}

\newcommand{\Z}{\mathbb{Z}}
\newcommand{\C}{\mathbb{C}}
\newcommand{\R}{\mathbb R}
\renewcommand{\H}{\mathbb{H}}

\newcommand{\abs}[1]{\lvert#1\rvert}

\newcommand{\rD}{\mathrm{D}}
\newcommand{\rE}{\mathrm{E}}
\newcommand{\rF}{\mathrm{F}}
\newcommand{\rK}{\mathrm{K}}

\newcommand{\la}{\langle}
\newcommand{\ra}{\rangle}

\newcommand{\be}{\begin {equation}}
\newcommand{\ee}{\end {equation}}
\newcommand{\bee}{\begin {equation*}}
\newcommand{\eee}{\end {equation*}}

\newcommand{\cf}{\emph{cf.}~}

\theoremstyle{Theorem}

\theoremstyle{Theorem}

\newtheorem{thm}{Theorem}[section]

\newtheorem{prpt}[thm]{Proposition}

\theoremstyle{Theorem}
\newtheorem{lem}{Lemma}[section]
\newtheorem{corl}[lem]{Corollary}
\newtheorem{thml}[lem]{Theorem}

\newtheorem{prpl}[lem]{Proposition}

\theoremstyle{Theorem}

\theoremstyle{Plain}

\theoremstyle{Definition}
\newtheorem{defn}{Definition}[section]

\newtheorem{dfnl}[lem]{Definition}

\theoremstyle{Theorem}
\newtheorem{thmd}[defn]{Theorem}

\newtheorem{lemd}[defn]{Lemma}
\newtheorem{prpd}[defn]{Proposition}
\newtheorem{conjd}[defn]{Conjecture}

\title[Conservation relations]{Conservation relations for local theta correspondence}

\author{Binyong Sun}
\address{Hua Loo-Keng Key Laboratory of Mathematics, Institute of Mathematics, AMSS\\
Chinese Academy of Sciences\\
Beijing, 100190, P.R. China} \email{sun@math.ac.cn}

\author{Chen-Bo Zhu}
\address{Department of Mathematics\\
National University of Singapore\\
Block S17, 10 Lower Kent Ridge Road, Singapore 119076}
\email{matzhucb@nus.edu.sg}

\begin{document}

\subjclass[2010]{22E46, 22E50 (Primary)} \keywords{local theta
correspondence, conservation relation, oscillator representation,
degenerate principal series}

\dedicatory{To the memory of Stephen Rallis}

\begin{abstract} We prove Kudla-Rallis conjecture on first occurrences of
local theta correspondence, for all irreducible dual pairs of type I
and all local fields of characteristic zero.
\end{abstract}

\maketitle

\section{Introduction}

The main goal of this article is to prove ``conservation
relations" for local theta correspondence, which was first conjectured by
Kudla and Rallis in the mid 1990's.

\subsection{Dual pairs of type I}
\label{sub1.1}

Fix a triple $(\rF, \rD, \epsilon)$ where $\epsilon=\pm 1$; $\rF$ is a local field of characteristic zero; and $\rD$ is either $\rF$, or a quadratic field extension of $\rF$, or a
central division quaternion algebra over $\rF$. Denote by $\iota$
the involutive anti-automorphism of $\mathrm D$ which is
respectively the identity map, the non-trivial Galois element, or
the main involution.

Let $U$ be an $\epsilon$-Hermitian right
$\mathrm D$-vector space, namely, $U$ is a finite dimensional right
$\mathrm D$-vector space, equipped with a non-degenerate $\mathrm
F$-bilinear map
\[
  \la\,,\,\ra_U : U\times U\rightarrow \mathrm D
\]
satisfying
\[
  \la u,u'a\ra_U=\la u,u'\ra_U\, a,\quad u,u'\in U, \,a\in \mathrm D,
\]
and
\[
   \la u,u'\ra_U=\epsilon \la u',u\ra_U^\iota,\quad u,u'\in U.
\]
Denote by $\oG(U)$ the isometry group of $U$, namely the group of all
$\rD$-linear automorphisms of $U$ preserving the form
$\la\,,\,\ra_U$. It is naturally a locally compact topological group and is a so-called classical
group as summarized in Table \ref{tableg}.
\begin{table}[h]
\caption{The classical group $\oG(U)$}\label{tableg}
\centering 
\begin{tabular}{c c c c c c c} 
\hline
$\rD$ & \vline & $\rF$ & quadratic extension & quaternion algebra\\
\hline
$\epsilon=1$ & \vline & orthogonal group & unitary group  & quaternionic symplectic group\\ 
\hline 
$\epsilon=-1$ & \vline & symplectic group & unitary group  & quaternionic orthogonal group \\
\hline 
\end{tabular}
\label{table:nonlin} 
\end{table}

Let $V$ be a $-\epsilon$-Hermitian left $\mathrm D$-vector space,
defined in an analogous way. Denote by $\oG(V)$ the isometry group
of $V$. Following Howe \cite{Ho1}, we call $(\oG(U), \oG(V))$ an
irreducible dual pair of type I. The tensor product $U\otimes_\rD V$
is a symplectic space over $\rF$ under the bilinear form
\begin{equation}
\label{tensor-form}
  \la u\otimes v, u'\otimes v'\ra:=\frac{\la u,u'\ra_U \,\la v,v'\ra_V^\iota+\la
  v,v'\ra_V\,\la u,u'\ra_U^\iota }{2},\quad u,u'\in U,\,v,v'\in V.
\end{equation}
Here $\la \,,\,\ra_V$ denotes the underlying $-\epsilon$-Hermitian
form on $V$ (similar notations will be used without further
explanation).

\begin{defn}\label{defheis}
The Heisenberg group $\oH(W)$ attached to a (finite dimensional) symplectic space $W$ over $\rF$ is the topological group which equals $W\times \rF$ as a topological space, and whose group multiplication is given by
\[
  (w,t)(w',t'):=(w+w', t+t'+\la w,w'\ra_{W}), \quad
  w,w'\in W,\,\,t,t'\in \rF.
\]
\end{defn}

The group $\oG(U)\times \oG(V)$ acts continuously on the Heisenberg group  $\oH(U\otimes_\rD V)$ as group
automorphisms by
\be \label{actgh}
  (g,h)\cdot \left(\sum_{i=1}^r u_i\otimes v_i,\,t\right):=\left(\sum_{i=1}^r g(u_i)\otimes
h(v_i),\,t\right),
\ee
for all $g\in \oG(U)$, $h\in \oG(V)$, $r\geq 0$, $u_1, u_2, \cdots,
u_r\in U$,  $v_1, v_2, \cdots, v_r\in V$ and $t\in \rF$. Using this
action, we form the Jacobi group
\[
  \oJ(U,V):=(\oG(U)\times \oG(V))\ltimes \oH(U\otimes_\rD V).
\]
For the study of local theta correspondence, it is natural to introduce the following modification of $\oG(U)$:
\[
   \oGb(U):=\left\{
              \begin{array}{ll}
                \cover{\Sp}(U), & \hbox{if $U$ is a symplectic space, namely $(\rD,\epsilon)=(\rF, -1)$;} \\
                \oG(U), & \hbox{otherwise.}
              \end{array}
            \right.
\]
Here $\cover{\Sp}(U)$ denotes the metaplectic group: it is the unique topological central extension of the symplectic group $\Sp(U)$ by $\{\pm 1\}$ which does not split unless $U=0$ or $\rF \cong \C$. Define $\oGb(V)$ analogously. Using the covering homomorphism $\oGb(U)\times \oGb(V)\rightarrow \oG(U)\times \oG(V)$ and the action \eqref{actgh}, we form the modified Jacobi group
\[
  \oJb(U,V):=(\oGb(U)\times \oGb(V))\ltimes \oH(U\otimes_\rD V).
\]

\subsection{Local theta correspondence}
The center of the Heisenberg group $\oH(W)$ of Definition \ref{defheis} is obviously identified with $\rF$. Fix a non-trivial unitary character
$\psi :\rF\rightarrow \C^{\times}$ throughout the article. As usual, a superscript ``$^\times$" over a ring indicates its multiplicative group of invertible elements.

From now on until Section \ref{NOtrivial}, we assume that $\rF$ is non-archimedean.  The smooth version of Stone-von Neumann Theorem asserts the following:

\begin{thmd}\label{stone}(\cf \cite[2.I.2]{MVW})
For any symplectic space $W$ over $\rF$, there exists a unique (up to isomorphism) irreducible smooth
representation of $\oH(W)$ with central character $\psi$.
\end{thmd}

For any totally disconnected, locally
compact, Hausdorff topological space $Z$, denote by $\con^\infty(Z)$ the space of locally constant $\C$-valued functions on $Z$, and by $\CS(Z)$
the space of all functions in  $\con^\infty(Z)$ with compact support. Taking a complete polarization $W=X\oplus Y$ of a symplectic space $W$ over $\rF$, let $\oH(W)$ act on  $\CS(X)$ by
\be\label{acth}
  ((x_0+ y_0, t)\cdot \phi)(x):=\phi(x+x_0)\, \psi(t+\la 2x+x_0,
y_0\ra_W),
\ee
for all $\phi\in \CS(X)$, $x, x_0\in X$, $y_0\in Y$ and $t\in
\rF$. It is easy to check that this defines an  irreducible smooth
representation of $\oH(W)$ with central character $\psi$.

\begin{defn}\label{defso}
Let $J$ be a totally disconnected, locally compact, Hausdorff topological group. Assume that $J$ contains the Heisenberg group $\oH(W)$ as a closed normal subgroup,
where $W$ is a symplectic space over $\rF$. A smooth representation of $J$ is called a smooth oscillator representation (associated to $\psi$, unless otherwise specified) if its restriction to $\oH(W)$ is irreducible and has central character $\psi$.
\end{defn}

Dixmier's version of Schur's Lemma \cite[0.5.2]{Wa1} implies that smooth oscillator representations are unique up to twisting by characters:

\begin{lemd}\label{uniqueo}
For any smooth oscillator representations  $\omega$ and $\omega'$ of $J$ as in Definition \ref{defso}, there exists a  unique character $\chi$ on $J$ which
is trivial on $\oH(W)$ such that  $\omega'\cong \chi\otimes \omega$.
\end{lemd}

For the dual pair of a non-trivial symplectic group and an odd orthogonal group, smooth oscillator representations of $\oJ(U,V)$ do not exist. Nevertheless, in all cases, there exists a smooth oscillator representation $\omega_{U,V}$ of $\oJb(U,V)$ which is genuine in the following sense:
\begin{itemize}
 \item when $U$ is a symplectic space,  $\varepsilon_U$ acts through the scalar multiplication by $(-1)^{\dim V}$ in $\omega_{U,V}$; and likewise for $V$ when $V$ is a symplectic space.
\end{itemize}
Here and henceforth, for a symplectic space $U$, $\varepsilon_U$ denotes the non-trivial element of the kernel of the covering homomorphism  $\oGb(U)\rightarrow \oG(U)$.
See Lemma \ref{extension00}.

Denote by $\Irr(\oGb(U))$ the set of all isomorphism classes of
irreducible admissible smooth representations of $\oGb(U)$. Define
$\Irr(\oGb(V))$ similarly. For a genuine smooth oscillator representation $\omega_{U,V}$ of $\oJb(U,V)$, put
\begin{eqnarray*}
 \mathcal R_{\omega_{U,V}}(U) &:=& \{\pi\in \Irr(\oGb(U))\mid \Hom_{\oGb(U)}(\omega_{U,V}, \pi)\neq 0\},\\
\mathcal R_{\omega_{U,V}}(V) &:=& \{\pi'\in \Irr(\oGb(V))\mid \Hom_{\oGb(V)}(\omega_{U,V}, \pi')\neq 0\},
 \end{eqnarray*}
and
\[
  \mathcal R_{\omega_{U,V}}(U,V):=\{(\pi,\pi')\in \Irr(\oGb(U))\times \Irr(\oGb(V))\mid
  \Hom_{\oGb(U)\times \oGb(V)}(\omega_{U,V}, \pi\otimes \pi')\neq 0\}.
\]
The theory of local theta correspondence begins with the following Howe Duality Conjecture:

\begin{conjd}\label{howed}
The set $\mathcal R_{\omega_{U,V}}(U,V)$ is the graph of a
bijection between $\mathcal R_{\omega_{U,V}}(U)$ and $\mathcal R_{\omega_{U,V}}(V)$.
\end{conjd}

Waldspurger proves the above conjecture when $\rF$ has odd residue
characteristic \cite{Wald}. (We will not assume the Howe duality
conjecture, and thus there will not be any assumption on the residue
characteristic of $\rF$.) For the theory of local theta
correspondence, a basic question is occurrence, which is to
determine the sets $\mathcal R_{\omega_{U,V}}(U)$ and $\mathcal
R_{\omega_{U,V}}(V)$. By the symmetric role of $U$ and $V$, without
loss of generality we shall focus on the set $\mathcal
R_{\omega_{U,V}}(U)$.

\subsection{A prologue: Conservation relations in the Witt group}
Before proceeding to the conservation relations for local theta correspondence, it is instructive to explain certain relations in the Witt group, which is actually the conservation relations for local theta correspondence in the case $U=0$ (the zero space).

Denote by $\widehat \CW^+_0$ the commutative monoid (under
orthogonal direct sum) of all isometry classes of
$-\epsilon$-Hermitian left $\mathrm D$-vector spaces. When no
confusion is possible, we do not distinguish an element of $\widehat
\CW^+_0$ with a space which represents it. Recall that the split
rank of a space $V\in\widehat \CW^+_0$ is defined to be
\[
 \rank \,V:=\max \ \{\dim Y\mid Y\textrm{ is a totally isotropic $\rD$-subspace of $V$}\}.
\]
Denote by $\H$ the hyperbolic plane in $\widehat \CW^+_0$, namely the element of $\widehat \CW^+_0$ with dimension $2$ and split rank $1$. A subset of $\widehat \CW^+_0$ of the form
\[
 \mathbf t:=\{V_0, V_0+\H, V_0+2\H,\cdots\}
\]
is called a Witt tower in $\widehat \CW^+_0$, where $V_0$ is an anisotropic element in $\widehat \CW^+_0$, namely an element of split rank $0$. Define the anisotropic degree of $\mathbf t$ by
\[
   \deg \mathbf t:=\dim V_0.
\]

Denote by $\CW_0$ the set of all Witt towers in $\widehat \CW^+_0$.
This is a quotient set of $\widehat \CW^+_0$, and the monoid
structure on $\widehat \CW^+_0$ descends to a monoid structure on
$\CW_0$. The resulting monoid $\CW_0$ is in fact a finite abelian
group (its order is a power of $2$), which is called the Witt group
of $-\epsilon$-Hermitian left $\mathrm D$-vector spaces. See Section
\ref{secwg} for details.

Write $\mathrm d_{\rD,\epsilon}$ for the maximal dimension
of an anisotropic element in $\widehat \CW^+_0$:
\[
  \mathrm d_{\rD,\epsilon}:=\max \{\dim V_0\mid V_0\textrm{ is an anisotropic element of $\widehat \CW^+_0$}\}.
\]
Recall the following well-known result:
\begin{prpd}\label{danis} (\cf \cite[Chapters 5 and 10]{Sch})
One has that
\begin{equation}
\label{drhov0}
  \mathrm d_{\rD,\epsilon}=\left\{
            \begin{array}{ll}
              0, \ \ \ \ \ & \hbox{if $V$ is a symplectic space;}\smallskip \\
              1, \ \ \ \ \ & \hbox{if $V$ is a quaternionic Hermitian space;} \smallskip\\
              2, \ \ \ \ \ & \hbox{if $V$ is a Hermitian space or a skew-Hermition space;}\smallskip \\
              3, \ \ \ \ \ & \hbox{if $V$ is a quaternionic skew-Hermitian space;} \smallskip\\
              4, \ \ \ \ \ & \hbox{if $V$ is a symmetric bilinear space.} \\
            \end{array}
          \right.
\end{equation}
Moreover, there exists a unique element $V^\circ\in \widehat \CW^+_0$ which is anisotropic and has dimension $\mathrm d_{\rD,\epsilon}$; and every anisotropic element of $\widehat \CW^+_0$  is isometrically isomorphic to a subspace of $V^\circ$.
\end{prpd}
Here and henceforth,  we refer the various cases by the types of the spaces under consideration.  For example, ``$V$ is a symplectic space" means that $(\rD,\epsilon)=(\rF, 1)$.

By Proposition \ref{danis}, the Witt group $\CW_0$ has a unique element of  anisotropic degree $\mathrm d_{\rD,\epsilon}$. Denote it by $\mathbf t_0^\circ$ and call it the anti-split Witt tower in $\widehat \CW^+_0$. Note that $\CW_0$ is trivial when $V$ is a symplectic space. Except for this case, $\mathbf t_0^\circ$ has order $2$ in $\CW_0$. By considering the negative of the orthogonal complement of an anisotropic element of $\widehat \CW^+_0$ in $V^\circ$, Proposition \ref{danis} implies the following relations in the Witt group:

\begin{prpd}\label{conserv0} We have \[\deg \mathbf t_1+\deg \mathbf t_2=\mathrm d_{\rD,\epsilon}\]
for all $\mathbf t_1, \mathbf t_2\in \CW_0$ with difference $\mathbf t_0^\circ$.
\end{prpd}

\subsection{The generalized Witt group}
\label{sub1.3}
Note that the subset $\mathcal R_{\omega_{U,V}}(U)$ of $\Irr(\oGb(U))$ depends only on the restriction of $\omega_{U,V}$ to the subgroup
\[
  \oJb_U(V):=\oGb(U)\ltimes \oH(U\otimes_\rD V)\subset \oJb(U,V).
\]
For a fixed $\pi\in \Irr(\oGb(U))$, we shall consider occurrence of $\pi$ in $\omega_{U,V}$, or the membership of $\pi$ in $\mathcal
R_{\omega_{U,V}}(U)$, as $V$ vary. We introduce the following

\begin{defn}\label{defenh} An enhanced oscillator representation of $\oGb(U)$ is a pair $(V, \omega)$, where $V$ is a
$-\epsilon$-Hermitian left $\rD$-vector space, and $\omega$ is a
smooth oscillator representation of  $\oJb_U(V)$  which is genuine
in the following sense: if $U$ is a symplectic space, then
$\varepsilon_U$ acts through  the scalar multiplication by
$(-1)^{\dim V}$ in $\omega$. Two enhanced oscillator representations
$(V_1, \omega_1)$ and $(V_2,\omega_2)$ of  $\oGb(U)$ are said to be
isomorphic if there is an isometric isomorphism $V_1\cong V_2$ such
that $\omega_1$ is isomorphic to $\omega_2$ with respect to the
induced isomorphism $\oJb_U(V_1)\cong \oJb_U(V_2)$.
\end{defn}

Denote by $\widehat{\CW}^+_U$ the set of all isomorphism classes of enhanced
oscillator representations of $\oGb(U)$. The set $\widehat{\CW}^+_U$ has a
natural additive structure which makes it a commutative monoid:
\[
  (V_1, \omega_1)+(V_2, \omega_2):=(V_1\oplus V_2, \omega_1 \otimes\omega_2).
\]
Here $V_1\oplus V_2$ denotes the orthogonal direct sum, and the
tensor product $\omega_1 \otimes\omega_2$ carries the representation
of the group $\oJb_U(V_1\oplus V_2)$, as follows:
\[
  (g, (w_1+w_2,t_1+t_2))\cdot(\phi_1\otimes \phi_2):=((g,(w_1,t_1))\cdot\phi_1)\otimes ((g,(w_2,t_2))\cdot\phi_2),
\]
where $g\in \oGb(U)$, $w_i\in U\otimes_\rD V_i$, $t_i\in \rF$ and
$\phi_i\in \omega_i$ ($i=1,2$). As before, we shall not distinguish
an element of $\widehat{\CW}^+_U$ with an enhanced oscillator
representation which represents it.

Recall the hyperbolic plane $\H\in \widehat \CW^+_0$. We shall define the hyperbolic plane $\H_U$ in $\widehat{\CW}^+_U$ as follows.
The symplectic space $U\otimes_{\rD}\H$ has a complete polarization
\be \label{udd}
  U\otimes_{\rD} \H =U\otimes \mathrm e_1\oplus U\otimes \mathrm e_2,
\ee
where $\mathrm e_1, \mathrm e_2$ is a basis of $\H$ of isotropic vectors. Define the hyperbolic plane $\H_U$ in $\widehat{\CW}^+_U$ to be the enhanced oscillator representation
\[
\H_U:=(\H, \omega_U)\in \widehat{\CW}^+_U,
 \]
where $\omega_U$ is the representation of $\oJb_U(\H)$ on the space $\CS(U\otimes \mathrm e_1)$ so that the Heisenberg group $\oH(U\otimes_{\rD} \H)$ acts as in \eqref{acth} (for the complete polarization \eqref{udd}), and $\oGb(U)$ acts by
\[
  (\bar{g}\cdot\phi)(x\otimes \mathrm e_1):=\phi((g^{-1}(x))\otimes \mathrm e_1),\quad \bar{g}\in \oGb(U),\,\phi\in \CS(U\otimes \mathrm e_1), \,x\in U.
\]
Here $g$ denotes the image of $\bar{g}$ under the covering homomorphism $\oGb(U)\rightarrow \oG(U)$.

\begin{defn} Two elements $\sigma_1, \sigma_2\in \widehat{\CW}^+_U$ are said to be Witt equivalent if there are non-negative integers $m_1$ and $m_2$ such that the
equality
\[
  \sigma_1+ m_1\H_U=\sigma_2+m_2\,\H_U
\]
holds in the commutative monoid $\widehat{\CW}^+_U$. This defines an equivalence relation on $\widehat{\CW}^+_U$ whose equivalence
classes are called Witt towers in $\widehat{\CW}^+_U$.
\end{defn}

For $\sigma=(V,\omega)\in \widehat{\CW}^+_U$, we shall refer to the dimension
and the split rank of $V$ as the dimension and the split rank of
$\sigma$:
\[
  \dim \sigma:=\dim V\quad \textrm{ and }\quad \rank \,\sigma:=\rank \,V.
\]
Each Witt tower $\mathbf t\subset \widehat{\CW}^+_U$ has a unique anisotropic representative, namely an element of split rank $0$ (this is a consequence of Lemma \ref{uniqueo}). Write $\sigma_\mathbf t$ for this anisotropic element. Then
\[
  \mathbf t=\{\sigma_\mathbf t, \sigma_\mathbf t+\H_U, \sigma_\mathbf t+2\H_U,\cdots\}.
\]
We define the anisotropic degree of $\mathbf t$ to be
\begin{equation}
\label{def-deg}
  \deg \mathbf t:=\dim \sigma_\mathbf t.
\end{equation}

Write $\CW_U$ for the set of all Witt towers in $\widehat{\CW}^+_U$. Similar to the Witt group $\CW_0$, the additive structure on $\widehat{\CW}^+_U$ descends to an additive structure on $\CW_U$ which makes  $\CW_U$ an abelian group (see Proposition \ref{exactcw}, part (c)). There is a short exact sequence  (see Proposition \ref{exactcw}, part (d))
\[
   1\rightarrow (\oG(U))^*\rightarrow
   \CW_U\rightarrow
   \mathcal \CW_0\rightarrow 1.
\]
Here and henceforth, for all topological group $G$, $G^*:=\Hom(G,\C^\times)$ denotes the group of all characters on $G$. When $\oG(U)$ is a perfect group, we have identifications
\[
  \widehat \CW^+_U=\widehat \CW^+_0\quad\textrm{and}\quad \CW_U=\CW_0.
\]
(This includes the case when $U=0$, with clearly consistent notation.)

\subsection{First occurrence index and conservation relations}
\label{firsoc}

With the above notion, we may rephrase the question of occurrence as
follows: for a given $\sigma=(V,\omega)\in \widehat{\CW}^+_U$, one seeks to determine the set
\[
  \mathcal R_\sigma:=\{\pi\in \Irr(\oGb(U))\mid \Hom_{\oGb(U)}(\omega, \pi)\neq 0\}.
\]
A clear necessary condition for $\pi\in  \mathcal R_\sigma$ is that $\pi$ is genuine with respect to $\sigma$ in the following sense:
\begin{itemize}
\item
 if $U$ is a symplectic space, then $\varepsilon_U$ acts through the scalar multiplication by $(-1)^{\dim \sigma}$ in $\pi$.
 \end{itemize}

Let  $\pi\in \Irr(\oGb(U))$ and let $\mathbf t\in\CW_U$. Assume that $\pi$ is genuine with respect to
$\mathbf t$, namely $\pi$ is genuine with respect to some (and hence all) elements of
$\mathbf t$.
There are two basic properties concerning occurrence:
\begin{itemize}
\item Occurrence in the so-called stable range (see \cite{Li1} and \eqref{stablerange}):
\be \label{stableint}
  \textrm{for all $\sigma\in \mathbf t$, $\,$ if $\,\rank\, \sigma\geq \dim U$,$\,$ then $\,\pi\in \mathcal R_\sigma$.}
\ee
\item Kudla's persistence principle (\cite{Ku1}):
\[ \textrm{for all $\sigma_1,\sigma_2\in \mathbf t$,$\,$ if $\,\dim \sigma_1\leq \dim \sigma_2$, then $\mathcal R_{\sigma_1}\subset \mathcal R_{\sigma_2}$.}\]
\end{itemize}
Write $1_U$ for the trivial representation of $\oGb(U)$. We note that in our formulation, Kudla's persistence principle follows clearly from the fact that $\Hom_{\oGb(U)}(\omega_U, 1_U)\neq 0$. (Recall that $\omega_U$ denotes the underlying representation of the hyperbolic plane $\H_U\in \widehat \CW_U^+$.) Define the first occurrence
index
\begin{equation}
\label{Symp-index}
   \operatorname n_\mathbf t(\pi):=\min\{\dim \sigma\mid \sigma\in \mathbf t, \,\pi\in \mathcal R_\sigma\}.
\end{equation}
In view of the aforementioned two properties, the first occurrence
index is finite and is of clear interest.

Generalizing the anti-split Witt tower $\mathbf t_0^\circ\in \CW_0$, we will define in Section \ref{kk} the anti-split Witt tower $\mathbf t_U^\circ\in \CW_U$. When $U$ is a symmetric bilinear space,
$\mathbf t_U^\circ\in \CW_U=(\oO(U))^*$ is the sign character; when $U$ is a symplectic space or a quaternionic Hermitian space (in which case $\oG(U)$ is a perfect group), $\mathbf t_U^\circ \in \CW_U =\CW_0$ is identical to $\mathbf t_0^\circ$. In all cases, $\mathbf t_U^\circ$ is characterized by the equality
\begin{equation}\label{degst}
  \operatorname n_{\mathbf t_U^\circ}(1_U)=2\dim U+\mathrm d_{\rD,\epsilon}.
\end{equation}
The element $\mathbf t_U^\circ\in \CW_U$ has anisotropic degree $\mathrm d_{\rD,\epsilon}$, and has order $2$ unless $U$ is the zero symmetric bilinear space (in this exceptional case the group $\CW_U$ is trivial).

In the non-archimedean case, the conservation relations assert the
following:

\begin{thmd}\label{main}
Let $\mathbf t_1$ and $\mathbf t_2$ be
two Witt towers in $\CW_U$ with difference $\mathbf t_U^\circ$. Then for
any $\pi\in \Irr(\oGb(U))$ which is genuine with respect to
$\mathbf t_1$ (and hence genuine with respect to $\mathbf t_2$), one
has that
\[
   \operatorname n_{\mathbf t_1}(\pi)+\operatorname n_{\mathbf t_2}(\pi)=
   2\dim U+\mathrm d_{\rD,\epsilon}.
\]
\end{thmd}

\vsp

\noindent {\bf Remarks}: (a) For orthogonal-symplectic and
unitary-unitary dual pairs, the conservation relations were
conjectured by Kudla and Rallis in the mid 1990's \cite{Ku2, KR3}.
For quaternionic dual pairs, the conjectured statements first
appeared in Gan-Tantono \cite[Section 4]{GT}.

(b) For orthogonal-symplectic dual pairs and $\pi$
supercuspidal, the conservation relations were due to Kudla and
Rallis \cite{KR3}. This was later extended to all irreducible
dual pairs of type I by Minguez \cite{Mi}, again for $\pi$ supercuspidal.

(c) The inequality $\operatorname n_{\mathbf t_1}(\pi)+\operatorname
n_{\mathbf t_2}(\pi)\geq 2\dim U+\mathrm d_{\rD,\epsilon}$ is due to Rallis
\cite{Ra1,Ra2} and Kudla-Rallis \cite[Theorem 3.8]{KR3} (for
orthogonal-symplectic dual pairs), and Gong-Greni\'{e}
\cite[Theorem 1.8]{GG} (for unitary-unitary dual pairs, following an
earlier work of Harris-Kudla-Sweet \cite{HKS}). See also Gan-Ichino
\cite[Theorem 5.4]{GI}.

\vsp

We now comment on the organization of this article. In Section
\ref{sub1.6}, we explain the idea of the proof of Theorem
\ref{main}, with a special focus on the upper bound, namely the
inequality $\operatorname n_{\mathbf t_1}(\pi)+\operatorname
n_{\mathbf t_2}(\pi)\leq 2\dim U+\mathrm d_{\rD,\epsilon}$ (a key
contribution of the current article). We give its main argument for
the case when $U$ is a quaternionic Hermitian space. In Section
\ref{EOR}, we introduce the Grothendieck group $\widehat \CW_U$ of
the commutative monoid $\widehat \CW_U^+$, to be called the
generalized Witt-Grothendieck group, and we give their basic
properties. In Section \ref{kk}, we introduce the notion of Kudla
characters and the Kudla homomorphism, which are then used to
explicitly determine the generalized Witt-Grothendieck groups. A
main purpose for introducing these new notions is to formulate the
conservation relations for all irreducible dual pairs of type I in a
uniform and conceptually simple manner. This approach is justified,
and in our view appealing, due to the commonality of the underlying
principles, both in the occurrence and non-occurrence aspects. In
Section \ref{dpsdm}, we review the doubling method and some results
on the structure of degenerate principal series of $\oGb(U)$ for $U$
split, and prove the upper bound in the conservation relations.
Section \ref{NOtrivial} is devoted to the phenomenon of
non-occurrence of the trivial representation before stable range,
which is responsible for the lower bound in the conservation
relations. We follow the method of Rallis \cite{Ra1,Ra2} (which
treat the case of orthogonal groups). It is worth mentioning that
for the base case ($\dim U=1$, and $V$ anisotropic; see Lemma \ref{dim1}), proving
non-occurrence of the trivial representation again requires the use
of the doubling method.  In the final Section \ref{secarch}, we
discuss the conservation relations in the archimedean case. Then
three  different phenomena occur: the same conservation relations as
in the non-archimedean case hold when $U$ is a real or complex
symmetric bilinear space; no conservation relations hold when $U$ is
a complex symplectic space or a real quaternionic Hermitian space;
when $U$ is a real symplectic space, a complex Hermitian or
skew-Hermitian space, or a real quaternionic skew-Hermitian space, a
more involved version of the conservation relations hold.

\vsp

\noindent {\bf Acknowledgements}:
This work began in March, 2012 at the Institute for Mathematical Sciences (IMS), National
University of Singapore. The authors thank IMS for the support. The authors also thank Wee Teck Gan, Michael
Harris, Atsushi Ichino and Jian-Shu Li for their interest, as well
as comments on an earlier version of this article. A special thanks go to Dipendra Prasad for sending us the preprint of
Rallis \cite{Ra2}, in addition to his interest. Finally the authors would like to
thank the referee for his insightful comments and detailed suggestions. B. Sun was supported in part by the NSFC Grants 11222101 and 11321101, and
C.-B. Zhu by the MOE grant MOE2010-T2-2-113.

\section{About the proof of Theorem \ref{main}}
\label{sub1.6}

\subsection{The strategy of the proof}\label{secstra}

Let $\mathbf t_1$, $\mathbf t_2$ and $\pi$ be as in Theorem \ref{main}. As usual, we use a superscript ``$^\vee$" to indicate the contragredient representation of an admissible smooth representation. There are two equally important aspects of the conservation
relations, which respectively assert the non-occurrence (Proposition \ref{lower0}):
\begin{equation}\label{geq}
\operatorname n_{\mathbf t_1}(\pi)+\operatorname
n_{-\mathbf t_2}(\pi^\vee) \geq 2\dim U+\mathrm d_{\rD,\epsilon},
\end{equation}
and the occurrence (Corollary \ref{upper}):
\begin{equation}\label{leq}
  \operatorname n_{\mathbf t_1}(\pi)+\operatorname
n_{\mathbf t_2}(\pi)\leq 2\dim U+\mathrm d_{\rD,\epsilon}.
\end{equation}
Assuming both \eqref{geq} and \eqref{leq}, we then have
\begin{equation}\label{enequ}
  \left\{
  \begin{array}{rl}
    \operatorname n_{\mathbf t_1}(\pi)+\operatorname n_{-\mathbf t_2}(\pi^\vee)&\geq \,2\dim U+\mathrm d_{\rD,\epsilon};\smallskip  \\
    \operatorname n_{\mathbf t_2}(\pi)+\operatorname n_{-\mathbf t_1}(\pi^\vee)&\geq \,2\dim U+\mathrm d_{\rD,\epsilon};\smallskip \\
    \operatorname n_{\mathbf t_1}(\pi)+\operatorname n_{\mathbf t_2}(\pi)&\leq \,2\dim U+\mathrm d_{\rD,\epsilon};\smallskip  \\
    \operatorname n_{-\mathbf t_1}(\pi^\vee)+\operatorname n_{-\mathbf t_2}(\pi^\vee)&\leq \,2\dim U+\mathrm d_{\rD,\epsilon}.
  \end{array}
\right.
\end{equation}
This forces all the above inequalities to be equalities! In particular we arrive at Theorem \ref{main}.

For the non-occurrence, the key phenomenon is the late occurrence of
the trivial representation $1_U$ in the anti-split Witt tower $\mathbf t_U^\circ$ (Proposition \ref{trivialv-ind}):
\begin{equation}\label{nonocsp}
   \operatorname n_{\mathbf t_U^\circ}(1_U)\geq 2\dim U+\mathrm d_{\rD,\epsilon}.
\end{equation}
This asserts the vanishing of certain vector valued $\oGb(U)$-invariant distributions.  We show \eqref{nonocsp} following the method of Rallis \cite{Ra1, Ra2}.
The proof consists of reduction to the null cone, and
within the null cone, proving vanishing on small orbits and
homogeneity for the main orbits, and finally employing the Fourier
transform. This is long and involved, and the details are given in Section \ref{NOtrivial}.

 As is well-known,  \eqref{nonocsp} implies \eqref{geq}, as follows. Let $\sigma_1=(V_1, \omega_1)$ be the element of $\mathbf t_1$ so that $\operatorname n_{\mathbf t_1}(\pi)=\dim \sigma_1$. Likewise, let $\sigma_2=(V_2, \omega_2)$ be the element of $-\mathbf t_2$ so that $\operatorname n_{-\mathbf t_2}(\pi^\vee)=\dim \sigma_2$. Then
\[
   \Hom_{\oGb(U)}(\omega_1, \pi)\neq 0\quad\textrm{and}\quad \Hom_{\oGb(U)}(\omega_2, \pi^\vee)\neq 0.
\]
Consequently,
\[
   \Hom_{\oGb(U)}(\omega_1\otimes \omega_2, 1_U)\neq 0,
\]
which implies that
\[
   1_U\in \mathcal R_{\sigma_1+\sigma_2}\quad\textrm{and so}\quad \dim (\sigma_1+\sigma_2)\geq \operatorname n_{\mathbf t_1-\mathbf t_2}(1_U)=\operatorname n_{\mathbf t_U^\circ}(1_U).
\]
Therefore \eqref{nonocsp} implies that
\[
  \operatorname n_{\mathbf t_1}(\pi)+\operatorname n_{-\mathbf t_2}(\pi^\vee)=\dim \sigma_1+\dim \sigma_2=\dim (\sigma_1+\sigma_2)\geq 2\dim U+\mathrm d_{\rD,\epsilon}.
\]

\subsection{The methods of Kudla and Rallis}\label{kr}

For the occurrence, we use the doubling method, theory of local zeta
integrals, and critically the known structure of degenerate
principle series, which are in fact all part of the foundational
work of Kudla and Rallis \cite{KR3}.
To illustrate how the methods of Kudla and Rallis will lead to \eqref{leq}, we give the main argument for one special case, namely theta lifting from a quaternionic symplectic group to a quaternionic orthogonal group.

We thus assume that $U$ is a quaternionic Hermitian space. Then $\oG(U)$ is a perfect group, and $\widehat{\CW}^+_U=\widehat \CW^+_0$. For each space $V\in \widehat \CW^+_0$, define its  discriminant to be
\[
   \operatorname{disc} V:=\left(\prod_{i=1}^m \la e_i,e_i\ra_{V}\,\la e_i,e_i\ra_{V}^\iota\right) (\rF^\times)^2\in \rF^\times/(\rF^\times)^2,
   \]
where $e_1,e_2,\cdots, e_m$ is an orthogonal basis of $V$. This is independent of the choice of orthogonal basis.

For each $V\in \widehat{\CW}^+_0$, denote by $\omega_V$ the unique  (up to isomorphism) smooth oscillator representation of
\[
   \oJ_U(V):=\oG(U)\ltimes \oH(U \otimes _\rD V)\subset \oJ(U,V).
\]
Likewise, denote by $\omega^\square_V$ the unique (up to isomorphism) smooth oscillator representation of
\[
   \oJ_U^\square(V):=\oG(U^\square)\ltimes \oH(U^\square \otimes _\rD V)\subset \oJ(U^\square,V).
\]
Here
\be \label{squareu}
  U^\square:=U\oplus U^-,
\ee
and $U^-$ denotes the space $U$ equipped with the form scaled by $-1$. Put
\[
  U^\triangle:=\{(u,u)\mid u\in U\}.
  \]
It is a Lagrangian subspace of $U^\square$, namely, it is totally isotropic and $\dim U^\square=2\dim U^\triangle$. Consequently, $U^\triangle\otimes_\rD V$  is a Lagrangian subspace of the symplectic space $U^\square \otimes_\rD V$.

\begin{lem}\label{haaro} (\cite[Theorem 1]{Ca} and \cite[Theorem 9.1]{Ho1})
Let $X$ be a Lagrangian subspace of a symplectic space $W$ over
$\rF$, to be viewed as an abelian subgroup of the Heisenberg group
$\oH(W)$. Let $\omega$ be an irreducible smooth representation of
$\oH(W)$ with central character $\psi$. Then there exists a non-zero
linear functional on $\omega$ which is invariant under $X$, and such
a linear functional is unique up to scalar multiples.
\end{lem}
\begin{proof}
Using the realization of $\omega$ in \eqref{acth}, the lemma follows by the existence and uniqueness of Haar measure on $X$.
\end{proof}

Using Lemma \ref{haaro}, we obtain a non-zero linear functional $\lambda_V$ on $\omega^\square_V$ which is invariant under $U^\triangle\otimes_\rD V\subset \oH(U^\square \otimes_\rD V)$. Denote by $\oP(U^\triangle)$ the parabolic subgroup of $\oG(U^\square)$ stabilizing $U^\triangle$. The linear functional $\lambda _V$ has the following transformation property: (\cf \cite[Section 6]{Ya})
\[
   \lambda_V(p\cdot\phi)=\chi_V(p)\lambda_V(\phi) \quad\textrm{ for all } p\in \oP(U^\triangle),\, \phi\in \omega^\square _V,
\]
where $\chi_V(p)$ denotes the image of $p$ under the composition of
\[
  \oP(U^\triangle)\xrightarrow{\textrm{restriction}} \GL(U^\triangle)\xrightarrow{\textrm{reduced norm}} \rF^\times\xrightarrow{a\mapsto \left((-1)^{\dim V}\, \disc V,\, a\right)_\rF\,\abs{a}_\rF^{\dim V}} \C^\times.
\]
Here and as usual, $(\, ,\,)_\rF$ denotes the quadratic Hilbert symbol for $\rF$, and $\abs{\,\cdot\, }_\rF$ denotes the normalized absolute value for $\rF$.

For every character $\chi\in (\oP(U^\triangle))^*$, put
\[
  \operatorname I(\chi):=\{f\in \con^\infty(\oG(U^\square))\mid
f(pg)=\chi(p)\,f(g),\,p\in
\oP(U^\triangle),\,g\in \oG(U^\square)\}.
\]
It is a smooth representation of $\oG(U^\square)$ under right translations. This is a so-called degenerate principal series representation.

Let $\pi\in \Irr(\oG(U))$ be as in the Introduction. The theory of local zeta integrals \cite{PSR,LR} implies that
\be \label{zetan}
  \Hom_{\oG(U)}(\operatorname I(\chi), \pi)\neq 0.
\ee
Here $\oG(U)$ is identified with the subgroup of $\oG(U^\square)$   pointwise fixing $U^-$.

Via matrix coefficients, the linear functional $\lambda_V$ induces a $\oG(U^\square)$-intertwining linear map
\be \label{rquotient}
  \omega_V^\square \rightarrow \operatorname I(\chi_V),\quad \phi\mapsto(g\mapsto \lambda_V(g\cdot\phi)).
\ee
Denote by $\RQ_V$ the image of \eqref{rquotient}.  It is easy to see that (Lemma \ref{cn0})
\be \label{homog}
  \Hom_{\oG(U)}(\RQ_V,\pi)\neq 0\quad \textrm{ implies }\quad \Hom_{\oG(U)}( \omega_V,\pi)\neq 0.
\ee

Let $\mathbf t_1, \mathbf t_2\in \CW_0=\CW_U$ so that their difference is the anti-split Witt tower. Let $V_1$ be a space in $\mathbf t_1$ and  let $V_2$ be a space in $\mathbf t_2$. The following result about structure of degenerate principal series is critical for the conservation relations:
\begin{prpl} \cite[Theorem 1.4]{Ya}\label{strdg}
If $\dim V_1+\dim V_2=2\dim U+1$, then $\operatorname I(\chi_{V_1})/\RQ_{V_1}\cong \RQ_{V_2}$ as smooth representations of $\oG(U^\square)$.
\end{prpl}

Using Proposition \ref{strdg}, we get
\begin{prpl}\label{strdg1}
If $\dim V_1+\dim V_2=2\dim U+1$, then
\[
\Hom_{\oG(U)}(\omega_{V_1},\pi)\neq 0\quad\textrm{ or }\quad \Hom_{\oG(U)}(\omega_{V_2},\pi)\neq 0.
\]
\end{prpl}
\begin{proof}
By Proposition \ref{strdg}, \eqref{zetan} (for $\chi=\chi_{V_1}$) implies that
\[
\Hom_{\oG(U)}(\RQ_{V_1},\pi)\neq 0\quad\textrm{ or }\quad \Hom_{\oG(U)}(\RQ_{V_2},\pi)\neq 0.
\]
The proposition then follows by \eqref{homog}.
\end{proof}

\begin{lem}\label{strdg2}
If $\operatorname n_{\mathbf t_1}(\pi)=\deg \mathbf t_1$ or $\operatorname n_{\mathbf t_2}(\pi)=\deg \mathbf t_2$, then
\[
  \operatorname n_{\mathbf t_1}(\pi)+\operatorname
n_{\mathbf t_2}(\pi)\leq 2\dim U+3.
\]
\end{lem}
\begin{proof}
Without loss of generality, assume that $\operatorname n_{\mathbf t_1}(\pi)=\deg \mathbf t_1$. By \eqref{stableint}, we have
\[
  \operatorname n_{\mathbf t_2}(\pi)\leq 2\dim U+\deg \mathbf t_2.
\]
Therefore Proposition \ref{conserv0} implies that
\[
  \operatorname n_{\mathbf t_1}(\pi)+\operatorname
n_{\mathbf t_2}(\pi)\leq \deg \mathbf t_1+2\dim U+\deg \mathbf t_2=2\dim U+3.
\]
\end{proof}

We now finish the proof of the inequality \eqref{leq}. In view of Lemma \ref{strdg2}, we may assume that $\operatorname n_{\mathbf t_1}(\pi)>\deg(\mathbf t_1)$ and $\operatorname n_{\mathbf t_2}(\pi)>\deg(\mathbf t_2)$. Then $U\neq 0$. Assume that \eqref{leq} does not hold. Then
\[
  \operatorname n_{\mathbf t_1}(\pi)+\operatorname
n_{\mathbf t_2}(\pi)\geq  2\dim U+5.
\]
This implies that there is a space $V_1$ in $\mathbf t_1$ and a space $V_2$ in $\mathbf t_2$ such that
\be \label{dimvi}
  \dim V_i<\operatorname n_{\mathbf t_i}(\pi), \quad i=1,2,
\ee
and
\be\label{dimv1}
  \dim V_1+\dim V_2=2\dim U+1.
\ee
We get a contradiction since \eqref{dimvi} implies that
\[
\Hom_{\oG(U)}(\omega_{V_1},\pi)=0\quad\textrm{ and }\quad \Hom_{\oG(U)}(\omega_{V_2},\pi)=0,
\]
and by Proposition \ref{strdg1}, \eqref{dimv1} implies that
\[
\Hom_{\oG(U)}(\omega_{V_1},\pi)\neq 0\quad\textrm{ or }\quad \Hom_{\oG(U)}(\omega_{V_2},\pi)\neq 0.
\]
This proves that the inequality \eqref{leq} always holds.

\section{Generalized Witt-Grothendieck groups}
\label{EOR}

\subsection{The Witt-Grothendieck group}\label{secwg}

Recall from the Introduction the commutative monoid $\widehat
\CW^+_0$ of isometry classes of  $-\epsilon$-Hermitian left $\mathrm
D$-vector spaces.
\begin{lem}\cite[Chapter 7, Corollary 9.2 (i)]{Sch}\label{canc0}
The monoid $\widehat \CW^+_0$  is cancellative, namely, for all $V_1, V_2, V_3\in \widehat \CW^+_0$, if $V_1+V_3=V_2+V_3$, then $V_1=V_2$.
\end{lem}

 Denote by $\widehat \CW_0$ the Grothendieck group \cite[Chapter 2, Definition 1.2]{Sch}  of the commutative monoid $\widehat \CW^+_0$. By definition, it is an abelian group together with a monoid homomorphism $\mathrm j_0:  \widehat \CW^+_0\rightarrow \widehat \CW_0$ with the following universal property: for each abelian group $A$ and each monoid homomorphism $\varphi^+: \widehat \CW^+_0\rightarrow A$, there is a unique group homomorphism $\varphi: \widehat \CW_0\rightarrow A$ such that $\varphi\circ \mathrm j_0=\varphi^+$.
 The group  $\widehat \CW_0$ is called  the Witt-Grothendieck group  of $-\epsilon$-Hermitian left $\mathrm D$-vector
spaces. Note that Lemma \ref{canc0} implies that $\mathrm j_0$ is injective \cite[Chapter 2, Lemma 1.3]{Sch}. Via $\mathrm j_0$, we view  $\widehat \CW^+_0$ as a submonoid of $\widehat \CW_0$.

Denote by $\rE$ the center of $\rD$. Put
\[
  \Delta:=\Delta_{\rF,\rD,\epsilon}:=\left\{
                                                              \begin{array}{ll}
                                                                \{1\}, & \hbox{$\epsilon=1$, $\rD=\rF$;} \\
                                                                \rE_{\pm}^\times/\mathrm N^\times, & \hbox{$\epsilon=1$, $\rD$ is a quadratic extension;} \\
                                                                \rF^\times/\RN^\times, & \hbox{$\epsilon=1$, $\rD$ is a quaternion algebra;} \\
                                                                 \oHil(\rF), & \hbox{$\epsilon=-1$, $\rD=\rF$;} \\
                                                                \rF^\times/\RN^\times, & \hbox{$\epsilon=-1$, $\rD$ is a quadratic extension;} \\
                                                                \{1\}, & \hbox{$\epsilon=-1$, $\rD$ is a quaternion algebra,}
                                                              \end{array}
                                                            \right.
\]
where
\[
\left\{
  \begin{array}{l}
    \RN^\times:=\{a a^\iota\mid a\in \rE^\times\}\subset \rF^\times;\smallskip \\
    \rE_\pm^\times:=\{a\in \rE^\times\mid a^\iota=a \text{ or } -a\};\smallskip \\
    \oHil(\rF):=\frac{\rF^\times}{(\rF^\times)^2}\times \{\pm 1\},
  \end{array}
\right.
\]
and $\oHil(\rF)$ is viewed as a group with group multiplication
\be\label{atat}
   (a, t)(a',t'):=(a a', t\, t' (a,a')_\rF), \quad a, a'\in\rF^\times/(\rF^\times)^2, \ t, t'\in \{\pm 1\}.
\ee
Here and as usual, $(\,,\,)_\rF$ denotes the quadratic Hilbert symbol for $\rF$. Then we have a group homomorphism
\be \label{dimdisc}
\disc: \widehat \CW_0\rightarrow \Delta
\ee
such that for each $V\in \widehat \CW^+_0$,
\[
 \disc \, V:=\left\{
                                                              \begin{array}{ll}
                                                                1, & \hbox{$\epsilon=1$, $\rD=\rF$;}\smallskip \\
                                                                \left(\prod_{i=1}^m \la e_i,e_i\ra_{V}\right) \RN^\times, & \hbox{$\epsilon=1$, $\rD$ is a quadratic extension;} \smallskip \\
                                                               \left(\prod_{i=1}^m \la e_i,e_i\ra_{V}\,\la e_i,e_i\ra_{V}^\iota\right) \RN^\times, & \hbox{$\epsilon=1$, $\rD$ is a quaternion algebra;} \smallskip \\
                                                                 (\left(\prod_{i=1}^m \la e_i,e_i\ra_{V}\right) \RN^\times, \operatorname{hass} V), & \hbox{$\epsilon=-1$, $\rD=\rF$;} \smallskip \\
                                                               \left(\prod_{i=1}^m \la e_i,e_i\ra_{V}\right) \RN^\times, & \hbox{$\epsilon=-1$, $\rD$ is a quadratic extension;} \smallskip \\
                                                                1, & \hbox{$\epsilon=-1$, $\rD$ is a quaternion algebra,}
                                                              \end{array}
                                                            \right.
\]
where $e_1,e_2,\cdots, e_m$ is an orthogonal basis of $V$, and $\operatorname{hass} V$ denotes the Hasse invariant of $V$ when $V$ is a symmetric bilinear space:
\[
  \operatorname{hass} V:=\prod_{1\leq i<j\leq m}(\la e_i,e_i\ra_{V}, \la e_j,e_j\ra_{V})_\rF.
\]

On the other hand, we also have the dimension homomorphism $\dim:
\widehat \CW_0\rightarrow \Z$. We summarize the well-know results on
the classification of $-\epsilon$-Hermitian left $\mathrm D$-vector
spaces as follows:
\begin{thml}\label{wg0} (\cf \cite[Chapters 5 and 10]{Sch})
The  homomorphism
\be\label{homcw0}
\dim\times \disc: \widehat \CW_0\rightarrow \Z\times \Delta
\ee
 is injective and its image equals the group of Table \ref{tablecw0}. If we identify $\widehat \CW_0$ with the image of \eqref{homcw0}, then an element $(m,\delta)$ of $\widehat \CW_0$ belongs to the monoid $\widehat \CW_0^+$ if and only if $m\geq 0$, and
\[
\left\{
  \begin{array}{ll}
  \delta=1,&\hbox{if $m=0$,}\\
   \delta\neq -1, & \hbox{if $m=1$, $\rD$ is a quaternion algebra and $\epsilon=1$}; \\
   \delta\in (\rF^\times/\RN^\times)\times\{1\}\subset \oHil(\rF), & \hbox{if $m=1$, $\rD=\rF$ and $\epsilon=-1$.}
  \end{array}
\right.
\]

\end{thml}

\begin{table}[h]
\caption{The Witt-Grothendieck group $\widehat \CW_0$}\label{tablecw0}
\centering 
\begin{tabular}{c c c c c c c} 
\hline
$\rD$ & \vline & $\rF$ & quadratic extension & quaternion algebra\\
\hline
 $\epsilon=1$ & \vline & $2\Z$ & $\Z\times_{\Z/2\Z} \frac{\rE_{\pm}^\times}{\mathrm N^\times}$  & $\phantom{\frac{\overbrace a}{\underbrace a}}\Z\times \frac{\rF^\times}{\RN^\times}\phantom{\frac{\overbrace a}{a}}$\\ 
\hline
$\epsilon=-1$ & \vline &$\Z\times \oHil(\rF)$ & $\Z\times \frac{\rF^\times}{\mathrm N^\times}$  & $\phantom{\frac{\overbrace a}{\underbrace a}}\Z \phantom{\frac{\overbrace a}{a}}$  \\
\hline  
\end{tabular}
\label{table:nonlin} 
\end{table}

The fiber product $\Z\times_{\Z/2\Z} \frac{\rE_{\pm}^\times}{\mathrm N^\times}$ of Table \ref{tablecw0} is defined with respect to the quotient homomorphism $\Z\rightarrow \Z/2\Z$, and the homomorphism $\frac{\rE_{\pm}^\times}{\mathrm N^\times}\rightarrow \Z/2\Z$ with kernel $\frac{\rF^\times}{\RN^\times}$.

\subsection{Commutator quotient groups of classical groups}
\label{commutator}

Let $U$ be an $\epsilon$-Hermitian right $\rD$-vector space as before. Then the commutator group $[\oG(U),\oG(U)]$ is closed in $\oG(U)$ (\cf \cite{Ri}). Put
\[
 \RA(U):=\oG(U)/[\oG(U),\oG(U)],
\]
which is an abelian topological group. Given a $\rD$-linear isometric embedding
\[
 \varphi: U\rightarrow U'
\]
of $\epsilon$-Hermitian right $\rD$-vector spaces, we have a continuous group homomorphism
\be \label{ogvarphi}
  \oG(\varphi): \oG(U)\rightarrow \oG(U')
\ee
such that for all $g\in \oG(G)$,
\[
   (\oG(\varphi)(g))(\varphi(u)+u')=\varphi(g(u))+u',
\]
for all $u\in U$, and all $u'\in U'$ which is perpendicular to $\varphi(U)$. The homomorphism \eqref{ogvarphi} descends to a continuous group homomorphism
\be \label{oavarphi}
  \RA(\varphi): \RA(U)\rightarrow \RA(U').
\ee
The assignments
\[
  U\mapsto \RA(U),\qquad  \varphi\mapsto \RA(\varphi)
\]
is a functor from the category of $\epsilon$-Hermitian right $\rD$-vector spaces (the morphisms in this category are $\rD$-linear isometric embeddings) to the category of abelian topological groups. Write $\RA_\infty$ for the direct limit of this functor:
\[
  \RA_\infty:=\varinjlim_U \RA(U).
\]
Recall that by definition of the direct limit, $\RA_\infty$  is an abelian  topological group together with a continuous homomorphism $\RA(U)\rightarrow \RA_\infty$ for each  $\epsilon$-Hermitian right $\rD$-vector space $U$, satisfying certain universal properties.

We summarize the well-known results on commutator quotient groups of classical groups as follows:

\begin{thml}\label{cq} (\cf \cite[Section 2]{Wall})
The topological group  $\RA_\infty$ is canonically isomorphic to the topological group of Table \ref{tnu}.
Furthermore the natural homomorphism $\RA(U)\rightarrow \RA_\infty$ is a topological isomorphism in the following cases:
\begin{itemize}
  \item[(i)] $U$ is a symplectic space or a quaternionic Hermitian space;
  \item[(ii)] $U$ is a non-zero Hermitian or skew-Hermitian space;
  \item[(iii)] $U$ is a non-anisotropic symmetric bilinear space or a non-anisotropic quaternionic skew-Hermitian space.
\end{itemize}
\end{thml}

\begin{table}[h]
\caption{The commutator quotient group $\RA_\infty$ }\label{tnu}
\centering 
\begin{tabular}{c c c c c c c} 
\hline
$\rD$ & \vline & $\rF$ & quadratic extension & quaternion algebra\\
\hline
 $\epsilon=1$ & \vline & $\frac{\rF^\times}{\RN^\times}\times \{\pm 1\}$ & $\frac{\RE^\times}{\RF^\times}$  & $\phantom{\frac{\overbrace a}{\underbrace a}}\{1\}\phantom{\frac{\overbrace a}{a}}$\\ 
\hline
$\epsilon=-1$ & \vline &$\{1\}$ & $\frac{\RE^\times}{\RF^\times}$  & $\phantom{\frac{\overbrace a}{\underbrace a}}\frac{\rF^\times}{\RN^\times} \phantom{\frac{\overbrace a}{a}}$  \\
\hline  
\end{tabular}
\label{table:nonlin} 
\end{table}

Identify $\RA_\infty$ with the group of Table \ref{tnu}. Write
\be
\label{Ainfty}
  \mu_U: \oG(U)\rightarrow \RA_\infty
\ee
for the homomorphism which descends to the natural  homomorphism $\RA(U)\rightarrow \RA_\infty$. We describe the homomorphism  $\mu_U$ case by case in what follows.
\vsp

\noindent {\bf Case 1}: $U$ is a quaternionic Hermitian space or a symplectic space. In this case, $\mathrm A_{\infty}$ is trivial and hence $\mu_U$ is also trivial.

\vsp

\noindent {\bf Case 2}: $U$ is a Hermitian space or a skew-Hermitian space.
Hilbert's Theorem 90 implies that the map
\be \label{hilbert}
  \rE^\times/\rF^\times \rightarrow \oU(\RE),\quad  x \rF^\times \mapsto \frac{x}{x^\iota}
\ee
is a topological isomorphism, where $\oU(\RE):=\{a\in \RE^\times \mid a a^\iota=1\}$. The homomorphism $\mu_U$ is the composition of
\[
  \oG(U)\xrightarrow{\det}\oU(\RE)\xrightarrow{\textrm{the inverse of \eqref{hilbert}}} \rE^\times/\rF^\times.
\]

\vsp
\noindent {\bf Case 3}: $U$ is a symmetric bilinear space.  In this case, $\oG(U)$ is generated by reflections \cite[Proposition 8]{Di1}, namely elements of the form  $s_v$ such that
\[
  \left\{
    \begin{array}{l}
      s_v(v)=-v, \textrm{ and} \\
     s_v(u)=u\textrm{ for all element $u\in U$ which is perperdicular to $v$},
    \end{array}
  \right.
 \]
where $v$ is a non-isotropic vector in $U$. The homomorphism $\mu_U$ is determined by
\[
  \mu_U(s_v)=\left(\la v,v\ra_U\, \RN^\times,-1\right)
\]
for all reflections $s_v\in \oG(U)$ (\cf \cite[Section 55]{OM}).

\vsp
\noindent {\bf Case 4}: $U$ is a quaternionic skew-Hermitian space. In this case,  $\oG(U)$ is generated by quasi-symmetries \cite[Theorem 2]{Di2}, namely elements of the form $s_{v,a}$ such that
\[
  \left\{
    \begin{array}{l}
      s_{v,a}(v)=va, \textrm{ and} \\
     s_{v,a}(u)=u\textrm{ for all element $u\in U$ which is perperdicular to $v$},
    \end{array}
  \right.
 \]
where $v$ is a non-isotropic vector in $U$, and $a$ is an element of $\rD$ which commutes with $\la v,v\ra_U$ and satisfies that $a a^\iota=1$. The homomorphism $\mu_U$ is determined by
\[
  \mu_U(s_{v,a})=\left\{
                   \begin{array}{ll}
                     (1+a)(1+a^\iota)\RN^\times, & \hbox{if $a\neq -1$;} \\
                     \la v,v\ra_U\,\la v,v\ra_U^\iota\, \RN^\times, & \hbox{if $a=-1$}
                   \end{array}
                 \right.
\]
for all quasi-symmetries $s_{v,a}\in \oG(U)$ (\cf \cite[Corollary
1.5 and Proposition 1.6]{Ku2} and \cite[Proposition 6.5]{Ya}).

\subsection{The generalized Witt-Grothendieck group}
\label{semigp}

Let $U$ be an $\epsilon$-Hermitian right $\mathrm D$-vector space,
and let $V$ be a $-\epsilon$-Hermitian left $\rD$-vector space.

\begin{lem}\label{extension00}
There exists a smooth oscillator representation of $\oJb(U,V)$. Every smooth oscillator representation $\omega_{U,V}$ of $\oJb(U,V)$ is unitarizable, and has the following property: when $U$ is a non-zero symplectic space, $\varepsilon_U$ acts through the scalar multiplication by $(-1)^{\dim V}$ in $\omega_{U,V}$; and likewise for $V$ when $V$ is a non-zero symplectic space.
\end{lem}
\begin{proof}
Besides splitting of metaplectic covers \cite[Proposition 4.1]{Ku2},
the first assertion and the second part of the last assertion are
due to the fact that any two elements in a metaplectic group commute
with each other if their projections to the symplectic group commute
with each other \cite[Chapter 2, Lemma II.5]{MVW}. The
unitarizability of $\omega_{U,V}$ is due to the fact that all
characters on $\oG(U)$ and $\oG(V)$ are unitary (see Theorem
\ref{cq}).
\end{proof}

\begin{lem}\label{extension0}
Every smooth oscillator representation of $\oJb_U(V)$
extends to a smooth oscillator representation of $\oJb(U,V)$.
\end{lem}
\begin{proof}
Let $\omega$ be a smooth oscillator representation of $\oJb(U,V)$. By Lemma \ref{uniqueo}, every smooth oscillator representation of $\oJb_U(V)$ is of the form $\chi\otimes \omega|_{\oJb_U(V)}$ for some character $\chi$ of $\oJb_U(V)$ which is trivial on $\oH(U\otimes_\rD V)$. The lemma then follows as $\chi$ extends to a character of $\oJb(U,V)$.
\end{proof}

\begin{lem}\label{eq0}
Two genuine smooth oscillator representations $\omega_1$ and $\omega_2$ of $\oJb_U(V)$
are isomorphic if and only if $(V, \omega_1)$ and $(V,
\omega_2)$ are isomorphic as enhanced oscillator
representations  of $\oGb(U)$.
\end{lem}
\begin{proof} The ``only if" part is trivial. To prove the ``if" part, assume that $(V, \omega_1)$ and $(V,
\omega_2)$ are isomorphic as enhanced oscillator
representations. This amounts to saying that there is an element
$g\in \oG(V)$ and a linear
isomorphism $\varphi: \omega_1\rightarrow \omega_2$ such that the diagram
\[
  \begin{CD}
            \omega_1 @>\varphi >> \omega_2\\
            @V h VV           @VV {g_U(h)}V\\
          \omega_1 @>\varphi >> \omega_2\\
  \end{CD}
\]
commutes for every $h\in \oJb_U(V)$, where $g_U:
\oJb_U(V)\rightarrow \oJb_U(V)$ is the automorphism induced by
$g:V\rightarrow V$.

Using Lemma \ref{extension0}, we extend $\omega_2$ to a representation of $\oJb(U,V)$,
which we still denote by $\omega_2$. Let $\bar{g}$ be an element of $\oGb(V)$ which lifts $g$. Then
$g_U(h)=\bar{g}h{\bar{g}}^{-1}$ for every $h\in \oJb_U(V)$. Therefore the
diagram
\[
  \begin{CD}
            \omega_1 @>{\bar{g}}^{-1}\circ \varphi >> \omega_2\\
            @V h VV           @VV h V\\
          \omega_1 @>{\bar{g}}^{-1}\circ \varphi >> \omega_2\\
  \end{CD}
\]
commutes for every $h\in \oJb_U(V)$, and consequently, $\omega_1$ and
$\omega_2$ are isomorphic as representations of
$\oJb_U(V)$.
\end{proof}

Recall the monoid $\widehat \CW^+_U$ from the Introduction. Note that $\oJb_U(0)=\oGb(U)\times \rF$. The map
\be\label{incch}
   (\oG(U))^*\rightarrow \widehat \CW^+_U,\qquad \chi\mapsto (0, \bar{\chi}\otimes \psi)
\ee
is clearly an injective monoid homomorphism, where $\bar{\chi}\in (\oGb(U))^*$ denotes the pull-back of $\chi$ through the covering homomorphism $\oGb(U)\rightarrow \oG(U)$.
Using this homomorphism, we view $(\oG(U))^*$ as a submonoid of $\widehat \CW^+_U$. Define a monoid homomorphism
\be \label{puplus}
   \widehat q_U^+: \widehat \CW^+_U\rightarrow \widehat \CW_0^+,\qquad (V,\omega)\mapsto V.
\ee
It is surjective by Lemma \ref{extension00}.

\begin{lem}\label{actst}
For each $V\in \widehat \CW_0^+$, the fiber $(\widehat q_U^+)^{-1}(V)$ is a principal homogeneous space of $(\oG(U))^*$ under the additive structure on $\widehat \CW^+_U$, that is, the map
\be \label{actp}
  +:  (\oG(U))^*\times  (\widehat q_U^+)^{-1}(V)\rightarrow (\widehat q_U^+)^{-1}(V)
\ee
is a well-defined  simply transitive action of $(\oG(U))^*$.
\end{lem}
\begin{proof}
The map \eqref{actp} is clearly a well-defined group action. Lemma \ref{uniqueo} implies that this action is transitive. Lemma \ref{uniqueo} and Lemma \ref{eq0} further imply that it is simply transitive.
\end{proof}

\begin{lem}\label{cancU}
The monoid $\widehat \CW^+_U$  is cancellative, namely, for all $\sigma_1, \sigma_2, \sigma_3\in \widehat \CW^+_U$, if $\sigma_1+\sigma_3=\sigma_2+\sigma_3$, then $\sigma_1=\sigma_2$.
\end{lem}
\begin{proof}
Applying the homomorphism $\widehat q_U^+$ to the equality $\sigma_1+\sigma_3=\sigma_2+\sigma_3$  in the lemma, Lemma \ref{canc0} implies that
\[
  \widehat q_U^+(\sigma_1)=\widehat q_U^+(\sigma_2).
\]
Then Lemma \ref{actst} implies that $\sigma_2=\chi+\sigma_1$ for some $\chi\in (\oG(U))^*$. Therefore
\[
  \chi+(\sigma_1+\sigma_3)=\sigma_2+\sigma_3=\sigma_1+\sigma_3,
\]
and Lemma \ref{actst} implies that $\chi$ is trivial. This proves that $\sigma_2=\sigma_1$.
\end{proof}

\begin{lem}\label{contra0}
Let $\sigma_1=(V_1,\omega_1)$, $\sigma_2=(V_2, \omega_2)\in \widehat \CW^+_U$. If $V_2$ is isometrically isomorphic to a subspace of $V_1$, then there exists a unique element $\sigma_3\in  \widehat \CW^+_U$ such that the equality
\[
  \sigma_2+\sigma_3=\sigma_1
\]
holds in $\widehat \CW^+_U$.
\end{lem}
\begin{proof}
The uniqueness is a direct consequence of Lemma \ref{cancU}. For existence, take an element $\sigma_3'=(V_3,\omega)$ so that $V_2+V_3=V_1$ in $\widehat \CW_0^+$. Then Lemma \ref{actst} implies that
\[
 \chi+(\sigma_2+\sigma_3')=\sigma_1
\]
for some $\chi\in (\oG(U))^*$. The lemma follows by putting $\sigma_3:=\chi+\sigma_3'$.
\end{proof}

Recall from the Introduction the hyperbolic plane  $\H\in \widehat \CW_0^+$ and the hyperbolic plane $\H_U\in \widehat \CW_U^+$. Lemma \ref{contra0} immediately implies the following
\begin{lem}\label{contra1}
For every $\sigma\in \widehat \CW^+_U$, there exists a unique element $\sigma^\vee\in  \widehat \CW^+_U$ such that the equality
\[
  \sigma+\sigma^\vee=(\dim \sigma )\,\H_U
\]
holds in $\widehat \CW^+_U$.
\end{lem}

As in Section \ref{secwg}, denote by $\widehat{\CW}_U$ the Grothendieck group of the commutative monoid $\widehat{\CW}^+_U$, and view $\widehat{\CW}^+_U$ as a submonoid of $\widehat{\CW}_U$ (by Lemma \ref{cancU}). The surjective  monoid homomorphism $\widehat{q}_U^+$ uniquely extends to a surjective group homomorphism
\be \label{hpu}
  \widehat{q}_U: \widehat \CW_U\rightarrow \widehat \CW_0.
\ee
Recall that $\CW_U$ denotes the set of Witt towers in $\widehat{\CW}^+_U$, which is a quotient monoid of $\widehat{\CW}^+_U$.

Using Lemmas \ref{actst}-\ref{contra1}, it is routine to prove the following proposition. We omit the details.
\begin{prpl}\label{exactcw}
The following hold true.

 \noindent
(a) The sequence
\be \label{exacthat}
  1\rightarrow (\oG(U))^*\rightarrow \widehat{\CW}_U \xrightarrow{\widehat q_U} \widehat{\CW}_0\rightarrow 1
\ee
is exact, and $\widehat q_U^{\,-1}(\widehat \CW_0^+)=\widehat \CW_U^+$.

\noindent
(b) Every element $\mathbf t\in \CW_U$ is uniquely of the form
\[
  \mathbf t=\{\sigma_\mathbf t+m\,\H_U\mid m=0,1,2,\cdots\},
\]
where $\sigma_\mathbf t$ is an anisotropic element of $\widehat \CW_U^+$, as defined in the Introduction.

\noindent
(c) The inclusion map $\widehat{\CW}^+_U\rightarrow \widehat{\CW}_U$ descends to a bijection $\CW_U\rightarrow \widehat{\CW}_U/(\Z\cdot \H_U)$. Consequently, the monoid  $\CW_U$ is an abelian group.

\noindent
(d)  The sequence
\[
  1\rightarrow (\oG(U))^*\xrightarrow{j_U} \CW_U \xrightarrow{q_U}\CW_0\rightarrow 1
\]
is exact, where $j_U$ is the restriction of the quotient map $\widehat{\CW}^+_U\rightarrow \CW_U$ to $(\oG(U))^*$, and $q_U$ is the descent of $\widehat q_U^+$.
\end{prpl}

Extending the notion of the split rank of an element of $\widehat \CW_U^+$, we define the split rank of each $\sigma\in \widehat \CW_U$,  which is denoted by $\rank \,\sigma$, to be the integer $m$ such  that  $\sigma-m\H_U$ is an anisotropic element of $\widehat \CW_U^+$. It is clear that
\[
  \sigma\in \widehat \CW_U^+\quad\textrm{if and only if}\quad \rank \,\sigma\geq 0.
\]

\subsection{Restrictions of enhanced oscillator representations}\label{reste}

Let $U'$ be an $\epsilon$-Hermitian right $\rD$-vector space with a $\rD$-linear isometric embedding
\[
     \varphi: U\hookrightarrow U'.
\]
For convenience, we identify $U$ with a subspace of $U'$ via $\varphi$. Denote by $U^\perp$ the orthogonal complement of $U$ in $U'$. Then
the induced embedding $\oG(U)\times \oG(U^\perp)\hookrightarrow \oG(U')$ uniquely lifts
to a ``genuine" homomorphism
\begin{equation}\label{embvvp1}
\oGb(U)\times \oGb(U^\perp)\rightarrow \oGb(U').
\end{equation}
Here ``genuine" means that when $U$ is a symplectic space, the homomorphism \eqref{embvvp1} maps both $\varepsilon_U$ and $\varepsilon_{U^\perp}$ to $\varepsilon_{U'}$.

For every  $-\epsilon$-Hermitian left $\rD$-vector space $V$,
combining \eqref{embvvp1} with the homomorphism
\[
\begin{array}{rcl}
  \oH(U\otimes_\rD V)\times \oH(U^\perp \otimes_\rD V)&\rightarrow &\oH(U'\otimes_\rD V),\\
    ((w,t),(w',t'))&\mapsto &(w+w',t+t'),
 \end{array}
\]
we obtain a homomorphism
\begin{equation}\label{embvvp2}
    \oJb_U(V)\times \oJb_{U^\perp}(V)\rightarrow\oJb_{U'}(V).
 \end{equation}

For each genuine smooth oscillator representation $\omega'$ of $\oJb_{U'}(V)$, its restriction through \eqref{embvvp2} is uniquely of the form
\begin{equation}\label{rdec}
  \omega'|_{\oJb_{U}(V)\times \oJb_{U^\perp}(V)}\cong\omega'|_U\otimes \omega'|_{U^\perp},
  \end{equation}
where $\omega'|_U$ is a genuine smooth oscillator representation of $\oJb_{U}(V)$, and $\omega'|_{U^\perp}$ is a genuine smooth oscillator representation of $\oJb_{U^\perp}(V)$. In turn this defines a monoid homomorphism (the restriction map)
\begin{equation}\label{embvvp3}
 \widehat{\mathrm r}_\varphi^+: \widehat \CW_{U'}^+\rightarrow \widehat \CW_{U}^+,\quad (V,\omega')\mapsto (V,\omega'|_U).
\end{equation}
It extends to a group homomorphism
\begin{equation}\label{embvvp35}
 \widehat{\mathrm r}_\varphi: \widehat \CW_{U'}\rightarrow \widehat \CW_{U}
\end{equation}
and descends to a group homomorphism
\begin{equation}\label{embvvp4}
  \mathrm r_\varphi: \CW_{U'}\rightarrow \CW_U.
\end{equation}

\begin{lem}\label{rewg}
The homomorphism $\widehat{\mathrm r}_\varphi^+: \widehat \CW_{U'}^+\rightarrow \widehat \CW_{U}^+$ only depends on $U$ and $U'$, that is, it does not depend on the $\rD$-linear isometric embedding $\varphi$. Consequently, both $\widehat{\mathrm r}_\varphi$ and $\mathrm r_\varphi$ do not depend on $\varphi$.
\end{lem}
\begin{proof} If $U'=U$, then $\varphi$ induces an inner automorphism $\oJb_U(V)\rightarrow \oJb_U(V)$, and consequently $\widehat{\mathrm r}_\varphi^+$ is the identity map. In general, the lemma follows by using  Witt's extension theorem \cite[Theorem 9.1]{Sch} and applying the above
argument to $U'$.
 \end{proof}

In view of Lemma \ref{rewg}, we also use $\widehat{\mathrm r}^{U'}_U$ and $\mathrm r^{U'}_U$ to denote $\widehat{\mathrm r}_\varphi$ and $\mathrm r_\varphi$, respectively. The following lemma is an obvious consequence of Proposition \ref{exactcw}.
\begin{lem}\label{decw}
One has that
\[
  \widehat \CW_U=\widehat{\mathrm r}^{U'}_U(\widehat \CW_{U'})+(\oG(U))^*\quad\textrm{and}\quad  \CW_U={\mathrm r}^{U'}_U(\CW_{U'})+(\oG(U))^*.
\]
\end{lem}

The assignments
\[
  U\mapsto \widehat \CW_{U}^+,\qquad  \varphi\mapsto \widehat{\mathrm r}_\varphi^+
\]
 form a contravariant functor from the category of $\epsilon$-Hermitian right $\rD$-vector spaces (the morphisms in this category are $\rD$-linear isometric embeddings) to the category of commutative monoids. Write $\widehat \CW^+_\infty$ for the inverse limit of this functor:
\[
  \widehat \CW^+_\infty:=\varprojlim_{U} \widehat \CW^+_U.
\]
Likewise, respectively using the homomorphisms of \eqref{embvvp35} and \eqref{embvvp4}, we form the inverse limits
\[
  \widehat \CW_\infty:=\varprojlim_{U} \widehat \CW_U\quad\textrm{ and }\quad  \CW_\infty:=\varprojlim_{U} \CW_U.
\]

Using the exact sequence of part (a) of Proposition \ref{exactcw}, the second assertion of Theorem \ref{cq} easily implies the following

\begin{lem}
The natural homomorphisms $\widehat \CW_\infty^+\rightarrow \widehat \CW_{U}^+$, $\widehat \CW_\infty\rightarrow \widehat \CW_{U}$ and $\CW_\infty\rightarrow \CW_{U}$ are isomorphisms in the following cases: \begin{itemize}
  \item[(i)] $U$ is a symplectic space or a quaternionic Hermitian space;
  \item[(ii)] $U$ is a non-zero Hermitian or skew-Hermitian space;
  \item[(iii)] $U$ is a non-anisotropic symmetric bilinear space or a non-anisotropic quaternionic skew-Hermitian space.
\end{itemize}
\end{lem}

We will explicitly determine the group $\widehat \CW_\infty$ (and hence the monoid $\widehat \CW_\infty^+$ and the group $\CW_\infty$) in all cases in the next section.

\section{Kudla characters and the Kudla homomorphism}
\label{kk}

We refer the reader to Section \ref{secwg} for the notations.
Define a compact abelian topological group
\[
  \RK:=\left\{
                  \begin{array}{ll}
                    \oHil(\rF), & \hbox{if $(\rD,\epsilon)=(\rF, -1)$;} \\
                   \RE^\times/\RN^\times, & \hbox{otherwise.}
                  \end{array}
                \right.
\]
In this section, we will define a (canonical) group homomorphism
\[
   \xi_\infty: \CW_\infty\rightarrow \RK^*.
\]
For $U$ split and non-zero, the related homomorphism
$\xi_U:\CW_U\rightarrow \RK^*$ regulates certain transformation
property of the ``Schr\"{o}dinger functional" (see Lemma \ref{haaro}) in
various Schr\"{o}dinger realizations of an enhanced oscillator
representation $(V,\omega)$, and in turn it determines a
$\bar{\oG}(U)$-intertwining map from $\omega$ to a degenerate
principal series representation of $\bar{\oG}(U)$. The definition is
inspired by the work of Kudla \cite{Ku2} on the splitting of
metaplectic covers and the pioneer work of Kudla-Rallis \cite{KR1}
on the structure of degenerate principal series representations. We
thus call $\xi_\infty(\mathbf t)\in \RK^*$ the Kudla character of a
Witt tower $\mathbf t\in \CW_\infty$, and $\xi_\infty$ the Kudla
homomorphism.

Moreover we will show (Corollary \ref {kgp0}) that $\xi_\infty$ is surjective and  $\ker \xi_\infty \cong \Z/2\Z$. Let $\mathbf t_{\infty}^\circ \in \CW_\infty$ be the non-trivial element of $\ker \xi_\infty$. For each $\epsilon$-Hermitian right $\rD$-vector space $U$, denote by $\mathbf t_U^\circ\in \CW_U$ the image of $\mathbf t_{\infty}^\circ$ under the natural homomorphism $\CW_\infty\rightarrow \CW_U$. This is the anti-split Witt tower in the Introduction.

\subsection{Some (coherent) characters on Siegel parabolic subgroups}
\label{k1}

Let $U$ be an $\epsilon$-Hermitian right $\rD$-vector space which is split in the sense that $\dim U=2\,\rank \,U$.  For every Lagrangian subspace $X$ of $U$, denote by $\oP(X)$ the (Siegel) parabolic subgroup of $\oG(U)$
stabilizing $X$, and by $\oPb(X)\rightarrow\oP(X)$ the covering homomorphism
induced by $\oGb(U)\rightarrow \oG(U)$.

We use $\abs{\,\cdot\,}_X$ to denote the following positive
character on $\oPb(X)$:
\begin{equation}\label{absx}
  \oPb(X)\rightarrow
   \oP(X)\xrightarrow{\textrm{restriction on $X$}}
\GL(X)\stackrel{\textrm{det}}{\longrightarrow}\rE^\times
\xrightarrow{\abs{\,\cdot\,}_\rE^{\delta_\rD}}\R^\times_+.
\end{equation}
Here and henceforth  ``$\det$" stands for the reduced
norm; $\abs{\,\cdot\,}_\rE$ is the normalized absolute value on
$\rE$; $\R^\times_+$ denotes the multiplicative group of positive real numbers; and $\delta_\rD$ is the degree of $\rD$ over $\rE$,
which is $2$ if $\rD$ is a quaternion algebra, and is $1$ otherwise.

Let $\sigma=(V,\omega)\in \widehat \CW_U^+$. By Lemma \ref{haaro},
there exists a unique (up to scalar multiples) non-zero linear
functional $\lambda_{X\otimes_\rD V}$ on $\omega$ which is invariant
under $X\otimes_\rD V\subset \oH(U\otimes_\rD V)$. Since $\oPb(X)$
normalizes $X\otimes_\rD V$ in $\oJb_U(V)$, there exists a character
$\kappa_{\sigma,X}$ on $\oPb(X)$ such that
\begin{equation}\label{lambdax}
  \lambda_{X\otimes_\rD V}(h\cdot \phi)=\kappa_{\sigma,X}(h)
  \abs{h}_X^{\frac{\dim \sigma}{2}}\lambda_{X\otimes_\rD V}(\phi),\quad h\in \oPb(X),\,\phi\in \omega.
\end{equation}

\begin{lem}\label{kudla}
Let $Y$ be another  Lagrangian subspace of $U$. Then
\[
  \kappa_{\sigma,X}(h)=\kappa_{\sigma,Y}(h)
\]
for all $h\in \oPb(X)\cap \oPb(Y)$.
\end{lem}

\begin{proof}
Using the Jordan decomposition, we assume without loss of generality
that $h$ is semisimple, namely the image $h_0$ of $h$ under the
covering homomorphism $\oGb(U)\rightarrow \oG(U)$ is semisimple.
Then it is elementary to see that there is an $h_0$-stable
Lagrangian subspace $Y'$ of $U$ such that
\[
  U=X\oplus Y'\quad\textrm{ and }\quad Y=(X\cap Y)\oplus (Y'\cap Y).
\]

Using the complete polarization
\[
  U\otimes_\rD V=(X\otimes_\rD V)\oplus (Y'\otimes_\rD V),
\]
we realize $\omega|_{\oH(U\otimes_\rD V)}$ on the space  $\CS(X\otimes_\rD V)$ as in \eqref{acth}. Respectively write $\mu_X$  and  $\mu_{X\cap Y}$ for Haar measures on $X\otimes_\rD V$ and $(X\cap Y)\otimes_\rD
V$. Then
\[
 \lambda_{X\otimes_\rD V}: \CS(X\otimes_\rD V)\rightarrow \C,\quad \phi\mapsto \int_{X\otimes_\rD V} \phi\,\mu_X
\]
is the unique (up to scalar multiples) linear functional on $\omega$ which is invariant under $X\otimes_\rD V$. Likewise
\[
 \lambda_{Y\otimes_\rD V}: \CS(X\otimes_\rD V)\rightarrow \C,\quad \phi \mapsto \int_{(X\cap Y)\otimes_\rD V} \phi|_{(X\cap Y)\otimes_\rD V}\, \mu_{X\cap Y}
\]
is the unique (up to scalar multiples) linear functional on $\omega$ which is invariant under $Y\otimes_\rD V$.

Note that there is a non-zero constant $c_h$ such that
\[
  (h\cdot\phi)(u)=c_h\,\phi((h_0^{-1}\otimes 1)(u))\quad \textrm{for all }\phi\in \CS(X\otimes_\rD V),\,u\in X\otimes_\rD V,
\]
where ``$1$" stands for the identity  map of $V$.
With the above explicit realizations, it is then routine to verify the equality in the lemma. We omit the details.
\end{proof}

Put
\begin{equation}
\label{split-part}
   \oGb(U)_{\mathrm{split}}:=\bigcup_{X\textrm{ is a  Lagrangian subspace of $U$}} \oPb(X).
\end{equation}
By Lemma \ref{kudla}, we get a well-defined map
\begin{equation*}
   \kappa_{\sigma}: \oGb(U)_{\mathrm{split}}\rightarrow \C^\times,
\end{equation*}
which sends $h\in \oPb(X)$ to $\kappa_{\sigma, X}(h)$, for each
Lagrangian subspace $X$ of $U$.

\begin{lem}
The map $\kappa_{\sigma}: \oGb(U)_{\mathrm{split}}\rightarrow \C^\times$ is $\oGb(U)$-conjugation invariant.
\end{lem}
\begin{proof}
Let $h\in \oPb(X)$ and let $g\in \oGb(U)$. Put $X':=g_0(X)$, where
$g_0$ denotes the image of $g$ under the covering homomorphism
$\oGb(U)\rightarrow \oG(U)$. Then the linear functional
\[
 \lambda': \omega\rightarrow \C, \quad \phi\mapsto \lambda_{X\otimes_\rD V}(g^{-1}\cdot \phi)
\]
is non-zero and invariant under $X'\otimes_\rD V\subset \oH(U\otimes_\rD V)$. The definition of $\kappa_{\sigma,X'}$ implies that
\[
  \lambda'((ghg^{-1})\cdot (g\cdot\phi))=\kappa_{\sigma,X'}(ghg^{-1})
  \abs{ghg^{-1}}_{X'}^{\frac{\dim \sigma}{2}}\lambda'(g\cdot \phi),\quad \phi\in \omega.
\]
This is equivalent to the following equality:
\be \label{lambdax2}
  \lambda_{X\otimes_\rD V}(h\cdot\phi)=\kappa_{\sigma,X'}(ghg^{-1})\abs{h}_{X}^{\frac{\dim \sigma}{2}}\lambda_{X\otimes_\rD V}(\phi),\quad \phi\in \omega.
\ee Therefore $\kappa_{\sigma}(ghg^{-1})=\kappa_{\sigma}(h)$, by
comparing \eqref{lambdax} and \eqref{lambdax2}.
\end{proof}

To summarize,  the map $\kappa_\sigma$  has the following two properties:
\begin{itemize}
  \item P1: it is $\oGb(U)$-conjugation invariant;
  \item P2: its restriction to $\oPb(X)$ is a continuous group homomorphism, for each Lagrangian subspace $X$ of $U$.
\end{itemize}
For any abelian topological group $A$, denote by
$\Hom(\oGb(U)_{\mathrm{split}}, A)$ the group of all maps from
$\oGb(U)_{\mathrm{split}}$ to $A$ with the properties P1 and P2. Thus
\begin{equation}
\label{kch} \kappa_\sigma\in \Hom(\oGb(U)_{\mathrm{split}},
\C^\times).
\end{equation}
It is easily verified that
\be \label{kappap}
  \widehat \CW_U^+ \rightarrow \Hom(\oGb(U)_{\mathrm{split}},
\C^\times), \quad \sigma\mapsto \kappa_\sigma
\ee
is a monoid homomorphism, and $\kappa_{\H_U}=1$.  Therefore the homomorphism \eqref{kappap} extends to a group homomorphism
\be \label{kappag}
  \widehat \CW_U \rightarrow \Hom(\oGb(U)_{\mathrm{split}},
\C^\times), \quad \sigma\mapsto \kappa_\sigma,
\ee
and descends to a group homomorphism
\be \label{kappag}
  \CW_U \rightarrow \Hom(\oGb(U)_{\mathrm{split}},
\C^\times), \quad \mathbf t\mapsto \kappa_{\mathbf t}.
\ee

\subsection{The group $\Hom(\oGb(U)_{\mathrm{split}}, \C^\times)$ and the Kudla homomorphism}
\label{k2} In this subsection, further assume that $U$ is non-zero.  We first work with the group $\oG(U)$. Put
\[
  \oG(U)_{\mathrm{split}}:=\bigcup_{X\textrm{ is a  Lagrangian subspace of $U$}} \oP(X).
\]
Similarly to $\Hom(\oGb(U)_{\mathrm{split}}, A)$, we define the group
$\Hom(\oG(U)_{\mathrm{split}}, A)$ for every abelian topological
group $A$.

Define a map
\[
  \nu_U: \oG(U)_{\mathrm{split}}\rightarrow \RE^\times /\RN^\times
\]
by sending $h\in \oP(X)$ to $\det(h|_X) \RN^\times$, for each
Lagrangian subspace $X$ of $U$. We omit the proof of the following
elementary lemma.

\begin{lem}\label{descg}
The map $\nu_U$ is well-defined, belongs to $\Hom(\oG(U)_{\mathrm{split}}, \RE^\times /\RN^\times)$, and is surjective.
The pull-back through $\nu_U$ yields a group isomorphism
\[\Hom(\oG(U)_{\mathrm{split}},A) \cong
  \Hom(\RE^\times /\RN^\times,A)
\]
for every abelian topological group $A$.
\end{lem}

Now assume that $U$ is a symplectic space. Then we have two natural maps (both two-to-one):
\[
 \widetilde \Sp(U)_{\mathrm{split}}:=\oGb(U)_{\mathrm{split}}\rightarrow \Sp(U)_{\mathrm{split}}:=\oG(U)_{\mathrm{split}}
 \]
and
\[
  \RK=\oHil(\rF)\rightarrow \rF^\times /\RN^\times,\quad (a,t)\mapsto a.
\]
Here the first map is obtained by restricting the  metaplectic cover. In what follows we define a map
 \[
  \widetilde \nu_U\in \Hom(\widetilde \Sp(U)_{\mathrm{split}}, \oHil(\rF))
  \]
 which lifts $\nu_U\in \Hom(\Sp(U)_{\mathrm{split}}, \rF^\times /\RN^\times)$.

Write $\sigma_\psi=(V_\psi, \omega_\psi)$ for the unique element of $\widehat \CW_U^+$ so that $V_\psi$ is one dimensional and has a vector $v_0$ such that $\la v_0,v_0\ra_{V_\psi}=1$. Applying the arguments of Section \ref{k1}, we get a function $\kappa_{\sigma_\psi}\in \Hom(\widetilde \Sp(U)_{\mathrm{split}}, \C^\times)$.

On the other hand, define a character $\gamma_\psi$ on $\oHil(\rF)$ by
\begin{equation}\label{gammap}
   \gamma_{\psi}(a,t):=t\,\frac{\gamma(x\mapsto \psi(ax^2))}{\gamma(x\mapsto \psi(x^2))}\in \C^\times, \quad (a,t)\in \oHil(\rF),
\end{equation}
where the two $\gamma$'s on the right hand side of \eqref{gammap} stand for Weil
indices (see \cite[Section 14]{Weil} or \cite{Weis}) of
non-degenerate characters (on $\rF$) of degree two.

\begin{lem}\label{ksp}
{\rm (a)} There exists a unique element $\widetilde \nu_U\in \Hom(\widetilde \Sp(U)_{\mathrm{split}}, \oHil(\rF))$
 which lifts $\nu_U$ and makes the diagram
\[
 \begin{CD}
            \widetilde \Sp(U)_{\mathrm{split}} @>\widetilde \nu_U >>  \oHil(\rF)\\
            @VV \kappa_{\sigma_\psi} V           @VV\gamma_{\psi} V\\
           \C^\times @= \C^\times\\
  \end{CD}
\]
commute. Moreover, $\widetilde \nu_U$ is independent of the choice
of the non-trivial unitary character $\psi$.

{\rm (b)} The map $\widetilde \nu_U$ is surjective, and via pull-back it yields a group isomorphism
\[
  \Hom(\widetilde \Sp(U)_{\mathrm{split}},A)\cong  \Hom( \oHil(\rF),A)
\]
for every abelian topological group $A$.
\end{lem}
\begin{proof} The uniqueness assertion of Part (a) is obvious. Let $\psi'$ be an arbitrary non-trivial unitary character of $\rF$.  Replacing the fixed character $\psi$ by $\psi'$ in the previous arguments, we get two functions
\[
  \kappa_{\sigma_{\psi'}}\in  \Hom(\widetilde \Sp(V)_{\mathrm{split}},\C^\times)\quad \textrm{and}\quad \gamma_{\psi'}: \oHil(\rF)\rightarrow \C^\times.
\]
It is known that \cite[Corollary A.5.]{Rao}
\begin{equation}\label{gg}
 \gamma_{\psi}(a,t)\, \gamma_{\psi'}(a,t)=(a,-\alpha)_\rF,\quad (a,t)\in \oHil(\rF),
\end{equation}
where $\alpha\in \rF^\times$ is determined by the formula
\[
  \psi'(x)=\psi(\alpha x),\quad x\in \rF.
\]

Let $h\in \widetilde \Sp(U)_{\mathrm{split}}$ and denote by $h_0$
its image under the map $\widetilde
\Sp(U)_{\mathrm{split}}\rightarrow \Sp(U)_{\mathrm{split}}$. By
using the Schr\"{o}dinger model for the dual pair of $\Sp(U)$ and an
even orthogonal group \cite[Section II.4]{Ku3}, we have that
\begin{equation}\label{wa4}
    \kappa_{\sigma_\psi}(h)\, \kappa_{\sigma_{\psi'}}(h)=(\nu_U(h_0), -\alpha)_\rF.
\end{equation}
Then it is elementary to see that \eqref{gg} and \eqref{wa4} imply that
\[
  \kappa_{\sigma_{\psi'}}(h)=\gamma_{\psi'}(\nu_U(h_0), t_h),
\]
where $t_h\in\{\pm 1\}$ is independent of $\psi'$. Put $\widetilde
\nu_U(h):=(\nu_U(h_0), t_h)$. It is then routine to check that
$\widetilde \nu_U\in \Hom(\widetilde \Sp(U)_{\mathrm{split}},
\oHil(\rF))$, and $\widetilde \nu_U$ has all properties of part (b)
of the lemma.
\end{proof}

\vvsp
We are back in the general case. Recall the compact abelian group $\RK$ from the beginning of this section. Define ${\bar \nu}_U\in \Hom(\oGb(U)_{\mathrm{split}},\RK)$ by
\begin{equation}
\label{muprime}
  {\bar \nu}_U:=\left\{
              \begin{array}{ll}
                \widetilde \nu_U, & \hbox{if  $U$ is a symplectic space;} \\
                \nu_U, & \hbox{otherwise.}
              \end{array}
            \right.
\end{equation}
For all cases, combining Lemmas \ref{descg} and \ref{ksp}, we get an isomorphism
 \begin{equation}
 \label{muinverse}
 ({\bar \nu_U})_*: \quad \Hom(\oGb(U)_{\mathrm{split}}, \C^\times) \overset{\sim}{\rightarrow} \rK^*.
 \end{equation}
We then have the homomorphism
\begin{equation}\label{kc}
  \xi_U: \CW_U\rightarrow \rK^*,\quad \mathbf t\mapsto ({\bar \nu}_U)_* (\kappa_\mathbf t).
\end{equation}

\begin{dfnl}
The Kudla homomorphism
\[
  \xi_\infty: \CW_\infty\rightarrow \RK^*
\]
 is the composition of $\xi_U$ with the natural isomorphism $\CW_\infty\rightarrow \CW_U$ (recall that $U$ is assumed to be split and non-zero in this subsection).
\end{dfnl}

\begin{lem}\label{k0}
The Kudla homomorphism $\xi_\infty$ is independent of the non-zero split
$\epsilon$-Hermitian right $\rD$-vector spaces $U$.
\end{lem}
\begin{proof}
Let $\varphi:U\rightarrow U'$ be a $\rD$-linear isometric embedding of non-zero split
$\epsilon$-Hermitian right $\rD$-vector spaces. Recall the restriction homomorphism $\mathrm r_\varphi: \CW_{U'}\rightarrow \CW_U$ from \eqref{embvvp4}. It suffices to show that $\xi_U \circ \mathrm r_\varphi =\xi_{U'}$.
This is a direct consequence of the fact that the
diagram
\[
  \begin{CD}
            \oGb(U)_{\mathrm{split}} @>>> \oGb(U')_{\mathrm{split}}\\
            @VV {{\bar \nu}_U}  V           @VV{\bar \nu}_{U'}  V\\
           \rK @= \rK\\
  \end{CD}
\]
commutes, where the top horizontal arrow is obtained by restricting the homomorphism \eqref{embvvp1}.
\end{proof}

\subsection{The group $\widehat \CW_\infty$}
\label{k3}

Recall the homomorphisms
\[
  \dim: \widehat \CW_0\rightarrow \Z\quad \textrm{and}\quad \disc: \widehat \CW_0\rightarrow \Delta.
\]
Still write
\[
  \dim: \widehat \CW_\infty\rightarrow \Z\quad \textrm{and}\quad \disc: \widehat \CW_\infty\rightarrow \Delta
\]
for their respective compositions with the natural homomorphism $\widehat \CW_\infty\rightarrow \widehat \CW_0$. Denote by
\be \label{kcw}
  \widehat \xi_\infty: \widehat \CW_\infty\rightarrow \RK^*
\ee
the compositions of $\xi_\infty: \CW_\infty\rightarrow \RK^*$ with the natural homomorphism $\widehat \CW_\infty\rightarrow \CW_\infty$.

\begin{thml}\label{cwinf}
The group $\widehat \CW_\infty$ is canonically isomorphic to the group of Table \ref{tablecwi}. Under this isomorphism, the homomorphisms $\dim: \widehat \CW_\infty\rightarrow \Z$ and $\disc: \widehat \CW_\infty\rightarrow \Delta$ are identical to the obvious projections, and $\widehat \xi _\infty: \widehat \CW_\infty\rightarrow \RK^*$ is identical to the obvious projection except for the following two cases:
\begin{itemize}
  \item[(i)] when $\epsilon=1$ and $\rD$ is a quaternion algebra, $\widehat \xi_\infty$ is identical to the homomorphism
\be \label{yamana}
  (m,\delta)\mapsto ((-1)^m\delta, \,\cdot\,)_\rF;
\ee
  \item[(ii)]  when $\epsilon=-1$ and $\rD=\rF$, $\widehat \xi_\infty$ is identical to the homomorphism
\be\label{homsym}
  (m,(\delta,t))\mapsto \left\{
                          \begin{array}{ll}
                            \left((-1)^\frac{m^2-m}{2} \delta, \, p_{\rF}(\,\cdot\,)\right)_\rF, & \hbox{if $m$ is even;} \\
                            \gamma_{\psi'}, & \hbox{if $m$ is odd,}
                          \end{array}
                        \right.
\ee
\end{itemize}
where $p_\rF:  \oHil(\rF)\rightarrow \rF^\times/\RN^\times$ is the projection map, $\psi'$ is the character of $\rF$ given by
\[
  \psi'(x):=\psi((-1)^\frac{m^2-m}{2} \delta x), \quad x\in\rF,
\]
and $\gamma_{\psi'}$ is defined as in \eqref{gammap}.
\end{thml}

\begin{table}[h]
\caption{The generalized Witt-Grothendieck group $\widehat \CW_\infty$}\label{tablecwi}
\centering 
\begin{tabular}{c c c c c c c} 
\hline
$\rD$ & \vline & $\rF$ & quadratic extension & quaternion algebra\\
\hline
 $\epsilon=1$ & \vline & $2\Z\times \left(\frac{\rF^\times}{\RN^\times}\right)^* \times \{\pm 1\}^*$ & $\left(\Z\times_{\Z/2\Z} \frac{\rE_{\pm}^\times}{\mathrm N^\times}\right)\times_{\Z/2\Z}\left(\frac{\rE^\times}{\RN^\times}\right)^*$  & $\phantom{\frac{\overbrace a}{\underbrace a}}\Z\times \frac{\rF^\times}{\RN^\times}\phantom{\frac{\overbrace a}{a}}$\\ 
\hline
$\epsilon=-1$ & \vline &$\Z\times \oHil(\rF)$ & $\left(\Z\times \frac{\rF^\times}{\mathrm N^\times}\right)\times_{\Z/2\Z} \left(\frac{\rE^\times}{\RN^\times}\right)^*$  & $\phantom{\frac{\overbrace a}{\underbrace a}}\Z \times\left(\frac{\rF^\times}{\RN^\times}\right)^*\phantom{\frac{\overbrace a}{a}}$  \\
\hline  
\end{tabular}
\label{table:nonlin} 
\end{table}

The fiber product $\Z\times_{\Z/2\Z} \frac{\rE_{\pm}^\times}{\mathrm N^\times}$ of  Table \ref{tablecwi} is the same as in Table \ref{tablecw0}.
For the data in the definitions of the other fiber products in Table \ref{tablecwi}, we are given the homomorphism $\left(\frac{\rE^\times}{\RN^\times}\right)^*\rightarrow \Z/2\Z$ whose kernel equals $\left(\frac{\rE^\times}{\rF^\times}\right)^*\subset \left(\frac{\rE^\times}{\RN^\times}\right)^*$, the  homomorphism $\Z\times_{\Z/2\Z} \frac{\rE_{\pm}^\times}{\mathrm N^\times}\rightarrow \Z/2\Z$ whose kernel equals $2\Z\times \frac{\rF^\times}{\mathrm N^\times}$, and the homomorphism $\Z\times \frac{\rF^\times}{\mathrm N^\times}\rightarrow \Z/2\Z$ whose kernel equals $2\Z\times \frac{\rF^\times}{\mathrm N^\times}$.

\vsp We prove Theorem \ref{cwinf} case by case in what follows.
Since there is a canonical isomorphism $\widehat \CW_\infty\cong
\widehat \CW_U$, there is no harm to replace $\widehat \CW_\infty$
by $\widehat \CW_U$ in the proof. Here $U$ is a non-zero split
$\epsilon$-Hermitian right $\rD$-vector space, as before. Write
\be\label{xiuh}
\widehat \xi_U: \widehat \CW_U\rightarrow \RK^* \ee for
the composition of $\xi_U: \CW_U\rightarrow \RK^*$ with the quotient
map $\widehat \CW_U\rightarrow \CW_U$.

Recall the homomorphism $\mu_U: \oG(U)\rightarrow \RA_\infty$ in \eqref{Ainfty} (Section \ref{commutator}).

\vsp

\noindent {\bf Case 1}: $U$ is a symmetric bilinear space.
 We have that
\begin{eqnarray*}
  \widehat \CW_U &=&  \Z\,\H_U \oplus (\oG(U))^*\qquad \qquad\,\,\,\, \textrm{by Proposition \ref{exactcw}}\\
   &=& 2\Z\oplus \left(\frac{\rF^\times}{\RN^\times} \times \{\pm 1\}\right)^* \qquad \textrm{by Theorem \ref{cq}}\\
   &=& 2\Z\times \left(\frac{\rF^\times}{\RN^\times}\right)^* \times \{\pm 1\}^*.
\end{eqnarray*}
Theorem \ref{cwinf} in this case then follows by noting that the diagram
\[
  \begin{CD}
            \oG(U)_{\mathrm{split}} @>\textrm{inclusion}>>  \oG(U)\\
            @VV {{\nu}_U}  V           @VV{\mu_U}  V\\
           \RK=\rF^\times/\RN^\times @>\textrm{inclusion}>> \RA _{\infty}=\rF^\times/ \RN^\times \times \{\pm 1\}\\
  \end{CD}
\]
commutes.

\vsp

\noindent {\bf Case 2}: $U$ is a symplectic space.  We have that
\begin{eqnarray*}
  \widehat \CW_U &=& \widehat \CW_0\qquad\quad \qquad \,\,\textrm{by Proposition \ref{exactcw} and Theorem \ref{cq}}\\
   &=& \Z\times \oHil (\rF). \qquad \textrm{by Theorem \ref{wg0}.}
   \end{eqnarray*}
Note that \eqref{gg} implies that \eqref{homsym} is a group homomorphism, and Lemma \ref{ksp} implies that $\widehat \xi_U$ and the map \eqref{homsym} are identical at all elements of $\widehat \CW_U^+$ of dimension one. Therefore Theorem \ref{cwinf} in this case follows.

\vsp

\noindent {\bf Case 3}: $U$ is a Hermitian space or a skew-Hermitian space. It follows from the discussion of \cite[Section 1]{HKS} that the image of
\begin{equation}\label{ktimesx}
  \widehat q_U\times \widehat \xi_U: \widehat \CW_U\rightarrow \widehat \CW_0\times (\RE^\times /\RN^\times)^*
 \end{equation}
is contained in the fibre product $\widehat \CW_0\times_{\Z/2\Z} (\RE^\times /\RN^\times)^*$.  In view of the exact sequence \eqref{exacthat}, Theorem \ref{cwinf} in this case follows by noting that the diagram
\[
  \begin{CD}
            (\oG(U))^* @>\textrm{the inclusion}>>  \widehat \CW_U\\
            @V \textrm{the isomorphism induced by $\mu_U$} V V           @VV \widehat \xi_U V\\
           \RA _{\infty}^*=(\RE^\times/\rF^\times)^* @>\textrm{the inclusion}>>   \RK ^*= (\RE^\times/\RN^\times)^*\\
  \end{CD}
\]
commutes.

\vsp

\noindent {\bf Case 4}: $U$ is a quaternionic Hermitian space.  We have that
\begin{eqnarray*}
  \widehat \CW_U &=& \widehat \CW_0\qquad \qquad \qquad \,\,\,\,\, \textrm{by Proposition \ref{exactcw} and Theorem \ref{cq}}\\
   &=& \Z\times (\rF^\times/\RN^\times). \qquad \textrm{by Theorem \ref{wg0}.}
   \end{eqnarray*}
By \cite[Section 6]{Ya}, we know that $\widehat \xi_\infty$ equals the map \eqref{yamana}.

\vsp

\noindent {\bf Case 5}:  $U$ is a quaternionic skew-Hermitian space. Note that the diagram
\[
  \begin{CD}
            (\oG(U))^* @>\textrm{the inclusion}>>  \widehat \CW_U\\
            @V \textrm{the isomorphism induced by $\mu_U$} V  V           @VV \widehat \xi_U V\\
           \RA _{\infty}^*=(\RF^\times/\RN^\times)^* @=   \RK ^*= (\RF^\times/\RN^\times)^*\\
  \end{CD}
\]
commutes.  Using  Theorem \ref{cq} and the exact sequence \eqref{exacthat}, this implies that the homomorphism
\[
  \dim \times \widehat \xi_U: \widehat \CW_U\rightarrow \Z\times (\RF^\times/\RN^\times)^*
\]
is an isomorphism, and Theorem \ref{cwinf} holds in this case.

\vsp

In conclusion, we have proved Theorem \ref{cwinf} in all cases. As a direct consequence of Theorem \ref{cwinf}, we have the following

\begin{corl}\label{kgp0}
The Kudla homomorphism $\xi_\infty: \CW_\infty\rightarrow \RK^*$ is
surjective and its kernel has order $2$. The non-trivial element of
the kernel has anisotropic degree (see \eqref{def-deg}) $\mathrm
d_{\rD,\epsilon}$.
\end{corl}


\section{Degenerate principal series and the doubling method}
\label{dpsdm}

\subsection{Degenerate principal series
representations}\label{dps}

Let $U$ be a split $\epsilon$-Hermitian right $\rD$-vector space,
with a Lagrangian subspace $X$. For each character $\chi\in
(\oPb(X))^*$, put
  \begin{equation}
\label{degene}
  \operatorname I(\chi):=\{f\in \con^\infty(\oGb(U))\mid
f(px)=\chi(p) f(x),\,p\in \oPb(X),\,x\in \oGb(U)\}.
\end{equation}
Under right translations, this is a smooth representation of
$\oGb(U)$.

Define a group homomorphism
\[
  \widehat \CW_U\rightarrow (\oPb(X))^*, \quad \sigma\mapsto \chi_{\sigma,X}
\]
so that
\[
  \chi_{\sigma,X}(h):=\kappa_{\sigma,X}(h)\abs{h}_X^{\frac{\dim \sigma}{2}},
  \quad h\in \oPb(X)
\]
for all $\sigma\in \widehat \CW_U^+$. See equation \eqref{lambdax}.

For each $\sigma\in \widehat \CW_U$, we associate an important
subrepresentation $\RQ_\sigma$ of $\operatorname I(\chi_{\sigma,X})$
as follows: if $\sigma\notin \widehat \CW_U^+$, we simply put
$\RQ_\sigma=0$; if $\sigma=(V,\omega)\in \widehat \CW_U^+$, we
define $\RQ_\sigma$ to be the image of the $\oGb(U)$-intertwining
linear map
\begin{equation}
\label{embedding}
   \begin{array}{rcl}
     \omega&\rightarrow &\operatorname I(\chi_{\sigma, X}),\\
     \phi&\mapsto & (g\mapsto \lambda_{X\otimes_\rD V}(g\cdot \phi)),
  \end{array}
\end{equation}
where  the functional $\lambda_{X\otimes_\rD V}$ is as in
\eqref{lambdax}. The following result is well-known (see
\cite[Theorem II.1.1]{Ra1} and  \cite[Chapter 3, Theorem
IV.7]{MVW}):

\begin{prpl}\label{coinv}
Let $\sigma=(V,\omega)\in \widehat \CW_U^+$. Extend $\omega$ to a
smooth oscillator representation of $\oJb(U,V)$ so that the
functional $\lambda_{X\otimes_\rD V}$ is $\oGb(V)$-invariant. Then
the homomorphism \eqref{embedding} descends to an isomorphism
$\omega_{\oGb(V)}\cong \RQ_\sigma$, where $\omega_{\oGb(V)}$ denotes
the maximal quotient of $\omega$  on which $\oGb(V)$ acts trivially.
\end{prpl}

The first key point of this article is the following proposition,
which is responsible for the upper bound in conservation relations.
\begin{prpl}\label{quotient0}
Let $\sigma_1$ and $\sigma_2$ be two elements of $\widehat \CW_U$ such that
\[
\left\{
  \begin{array}{l}
   \sigma_1-\sigma_2\textrm{ represents the anti-split Witt tower in $\CW_U$};\\
   \dim \sigma_1+\dim \sigma_2=\dim U+\mathrm d_{\rD, \epsilon}-2;\, \textrm{ and}\\
   \dim \sigma_1 \geq \dim \sigma_2.
  \end{array}
\right.
  \]
Then $\operatorname I(\chi_{\sigma_1,X})/\RQ_{\sigma_1}\cong
\RQ_{\sigma_2}$ as representations of $\oGb(U)$.
\end{prpl}
\begin{proof} The assertion follows from the work of Kudla-Rallis \cite[Introduction]{KR2}, Kudla-Sweet \cite[Theorem 1.2]{KS}, and Yamana \cite[Introduction]{Ya}.
\end{proof}
\noindent {\bf Remarks}: (a) We say that an element $\sigma\in \widehat \CW_U$ represents an element $\mathbf t\in \CW_U$ if $\sigma$ maps to $\mathbf t$ under the natural homomorphism $\widehat \CW_U\rightarrow \CW_U$ (we will apply this terminology to an arbitrary $\epsilon$-Hermitian right $\rD$-vector space, including the archimedean case).

(b) Write $\dim U=2r$ and put
\be
\label{defrhor}
  \rho_r =\frac{2r+\mathrm d_{\rD,\epsilon}-2}{4}.
\ee
Then $2\rho _r$ coincides with the normalized exponent of the modulus character of $\oP(X)$: the multiple of a left invariant Haar measure on $\oPb(X)$ by the function $\abs{\,\cdot\,}_X^{2\rho_r}$ (see \eqref{absx}) is a right invariant Haar measure on $\oPb(X)$. The condition $\dim \sigma+\dim \sigma'=\dim
U+\mathrm d_{\rD, \epsilon}-2$ of Proposition \ref{quotient0} then amounts to $\frac{\dim \sigma}{2}
+\frac{\dim \sigma'}{2}=2\rho _r$.

(c) Let $\sigma_1, \sigma_2\in \widehat \CW_U$. Assume that
\[
\left\{
  \begin{array}{l}
   \sigma_1-\sigma_2\textrm{ represents the anti-split Witt tower in $\CW_U$}; \textrm{ and}\\
   \dim \sigma_1+\dim \sigma_2=\dim U+\mathrm d_{\rD, \epsilon}-2.
  \end{array}
\right.
  \]
Then Proposition \ref{conserv0} implies that $\rank\, \sigma_1\geq
\rank\, U$ if and only if $\sigma_2\notin \widehat \CW_U^+$.
Therefore Proposition \ref{quotient0} implies that
\be\label{stable0} \RQ_\sigma=\operatorname I(\chi_{\sigma,X})\quad
\textrm{for all $\sigma\in \widehat \CW_U^+$ such that $\rank
\,\sigma\geq \rank\, U$.} \ee

\subsection{The doubling method}\label{dm}
Now we allow $U$ to be non-split, that is, $U$ is an arbitrary
$\epsilon$-Hermitian right $\rD$-vector space. Put
\[
  U^\square:=U\oplus U^-
\]
as in \eqref{squareu}.  Then
\[
  U^\triangle:=\{(u,u)\mid u\in U\}
\]
is a Lagrangian subspace of $U^\square$. As in Section \ref{dps}, we
have a subgroup $\oPb(U^\triangle)$ of $\oGb(U^\square)$, and a
representation $\operatorname I(\chi)$ of $\oGb(U^\square)$ for each
character $\chi\in (\oPb(U^\triangle))^*$. As in \eqref{embvvp1}, we
have a natural homomorphism $\oGb(U)\times \oGb(U^-)\rightarrow
\oGb(U^\square)$.

The theory of local zeta integrals \cite{PSR,LR} implies the following

\begin{lem}\label{zeta} Let $\pi\in \Irr(\oGb(U))$ and let
$\chi\in (\oPb(U^\triangle))^*$. When $U$ is a symplectic space,
assume that $\varepsilon_U$ acts through the scalar multiplication
by $\chi(\varepsilon_{U^\square})$ in $\pi$. Then
\[
  \Hom_{\oGb(U)}(\operatorname{I}(\chi), \pi)\neq 0.
\]
\end{lem}

For each $\sigma\in \widehat \CW_U$, put
\[
  \mathcal R_\sigma:=\left\{
                       \begin{array}{ll}
                         \{\pi\in \Irr(\oGb(U))\mid \Hom_{\oGb(U)}(\omega, \pi)\neq 0\}, & \hbox{if $\sigma=(V,\omega)\in \widehat \CW_U^+$;} \\
                         \emptyset, & \hbox{if $\sigma\notin \widehat \CW_U^+$.}
                       \end{array}
                     \right.
\]
The following result gives a sufficient
condition for non-vanishing of theta lifting.

\begin{lem}\label{cn0}
Let $\sigma^\square\in \widehat \CW_{U^\square}$ and let $\sigma:=\widehat{\mathrm r}^{U^\square}_U(\sigma^\square)\in \widehat \CW_{U}$ (see Lemma \ref{rewg}).
Then for all $\pi\in \Irr(\oGb(U))$,
\[
  \Hom_{\oGb(U)}(\RQ_{\sigma^\square}, \pi)\neq 0\quad \textrm{implies}\quad \pi\in \CR_\sigma.
\]
Here $\RQ_{\sigma^\square}$ is a subrepresentation of the
representation $\operatorname I(\chi_{\sigma^\square, U^\triangle})$
of $\oGb(U^\square)$, as in Section \ref{dps}.
\end{lem}
\begin{proof}
Write  $\sigma^\square=(V, \omega^\square)$.  Then $\RQ_{\sigma^\square}$ is isomorphic to a quotient of $(\omega^\square)|_{\oGb(U^\square)}$. Therefore
\[
  \Hom_{\oGb(U)}(\RQ_{\sigma^\square}, \pi)\neq 0\quad \textrm{implies}\quad \Hom_{\oGb(U)}(\omega^\square, \pi)\neq 0.
\]
Write  $\sigma=(V, \omega)$. Then $(\omega^\square)|_{\oGb(U)}$ is isomorphic to a direct sum of smooth representations which are isomorphic to $\omega |_{\oGb(U)}$. Therefore
\[
  \Hom_{\oGb(U)}(\omega^\square, \pi)\neq 0\quad \textrm{implies}\quad \Hom_{\oGb(U)}(\omega, \pi)\neq 0.
\]
\end{proof}

On the other hand, we have
\begin{lem}\label{cn1}
Let $\sigma^\square\in \widehat \CW_{U^\square}^+$ and let $\sigma:=\widehat{\mathrm r}^{U^\square}_U(\sigma^\square)\in \widehat \CW_{U}^+$ (see Lemma \ref{rewg}). Assume that $\sigma^\square$ is anisotropic and
$\xi_{U^\square}(\sigma^\square)$ is trivial (see \eqref{xiuh}).
Then for all $\pi\in \Irr(\oGb(U))$,
\[
 \pi\in \CR_\sigma \quad \textrm{implies}\quad \Hom_{\oGb(U)\times \oGb(U^-)}
 (\RQ_{\sigma^\square}, \pi \otimes  \pi^\vee)\neq 0.
\]
\end{lem}
\begin{proof} Write $\sigma^\square=(V,\omega^\square)$. As in \eqref{rdec}, write
\[
   \omega^\square|_{\oJb_U(V)\times \oJb_{U^-}(V)}=\omega \otimes \omega^-,
\]
where $\omega$ and $\omega^-$ are smooth oscillator representations
of $\oJb_U(V)$ and $\oJb_{U^-}(V)$, respectively. The triviality of
$\xi_{U^\square}(\sigma^\square)$  implies that $\omega$ and
$\omega^-$ are the contragredient representations of each other with
respect to the isomorphism
 \[
  \begin{array}{rcl}
  \oGb(U)\ltimes \oH(U\otimes_\rD V)&\rightarrow & \oGb(U^-)\ltimes
  \oH(U^-\otimes_\rD V),\\
 (g, (w,t))&\mapsto & (g, (w,-t)).
  \end{array}
\]
Extend $\omega$ and $\omega^-$ to representations of $\oJb(U,V)$ and
$\oJb(U^-,V)$, respectively, so that they are the contragredient
representations of each other with respect to the isomorphism
 \[
  \begin{array}{rcl}
  (\oGb(U)\times \oGb(V))\ltimes \oH(U\otimes_\rD V)&\rightarrow &
  (\oGb(U^-)\times \oGb(V))\ltimes  \oH(U^-\otimes_\rD V),\\
 ((g,h),(w,t))&\mapsto & ((g,h),(w,-t)).
  \end{array}
\]

Assume that $\pi\in \CR_\sigma$. Then there is an irreducible
representation $\tau\in \Irr(\oGb(V))$ such that (\cf \cite[Chapter
3, IV.4]{MVW})
\begin{equation}
\label{homnonzero}
  \Hom_{\oGb(U)\times \oGb(V)}(\omega, \pi\otimes \tau)\neq 0.
\end{equation}
Since $V$ is anisotropic, both $\pi$ and $\tau$ are unitarizable. By taking complex conjugations on the representations in \eqref{homnonzero}, we have that
\begin{equation}
\label{homnonzero2}
  \Hom_{\oGb(U^-)\times \oGb(V)}(\omega^-, \pi^\vee \otimes \tau^\vee)\neq 0.
\end{equation}
Combining \eqref{homnonzero} and \eqref{homnonzero2}, we have that
\[
    \Hom_{\oGb(U)\times \oGb(U^-)}((\omega \otimes \omega^-)_{\oGb(V)},
    \pi \otimes \pi^\vee)\neq 0,
\]
where a subscript ``${\oGb(V)}$" indicates the maximal quotient on
which $\oGb(V)$ acts trivially. The lemma then follows, by
Proposition \ref{coinv}.
\end{proof}

\noindent {\bf Remark}: The lemma above is a variant of a more
well-known result in the literature on local theta correspondence
(\cite[Proposition 3.1]{HKS} and \cite[Proposition 1.5]{Ku3}). Note
that we include the non-archimedean quaternionic case, for which
MVW-involutions do not exist \cite{LST}. To compensate this, the
space $V$ is assumed to be anisotropic, which is what we need (for
Lemma \ref{dim1}).

\subsection{Non-vanishing of theta lifting}\label{nonvs}

Concerning non-vanishing of theta lifting, we have
\begin{prpl}\label{dict0}
Let $\sigma_1$ and $\sigma_2$ be two elements of $\widehat \CW_U$ such that
\[
\left\{
  \begin{array}{l}
   \sigma_1-\sigma_2\textrm{ represents the anti-split Witt tower in $\CW_U$};\, \textrm{ and}\\
   \dim \sigma_1+\dim \sigma_2=2\dim U+\mathrm d_{\rD, \epsilon}-2.
  \end{array}
\right.
  \]
Then
\[
  \CR_{\sigma_1} \cup \CR_{\sigma_2}=\{\pi\in \Irr(\oGb(U))\mid \pi \textrm{ is genuine with respect to $\sigma_1$ (and $\sigma_2$)}\}.
\]
\end{prpl}

\begin{proof}
By Lemma \ref{decw} and without loss of generality, we assume that there exist $\sigma_i^\square\in
\widehat \CW_{U^\square}$  such that $\sigma_i=\widehat{\operatorname r}^{U^\square}_U(\sigma_i^\square)$ ($i=1,2$), and $\sigma_1^\square-\sigma_2^\square$ represents the anti-split Witt tower in $\CW_{U^\square}$.

By Proposition \ref{quotient0}, we either have a short exact sequence
\[
  0\rightarrow \RQ_{\sigma_1^\square}\rightarrow \operatorname I(\chi_{\sigma_1^\square,U^\triangle})\rightarrow \RQ_{\sigma_2^\square}\rightarrow 0,
\]
or have a short exact sequence
\[
  0\rightarrow \RQ_{\sigma_2^\square}\rightarrow \operatorname I(\chi_{\sigma_2^\square,U^\triangle})\rightarrow \RQ_{\sigma_1^\square}\rightarrow 0.
\]
Then Lemma \ref{zeta} implies that
\[
  \Hom_{\oGb(U)}(\RQ_{\sigma_1^\square},\pi)\neq 0\quad \textrm{or}\quad \Hom_{\oGb(U)}(\RQ_{\sigma_2^\square},\pi)\neq 0,
\]
for all $\pi\in  \Irr(\oGb(U))$ which is genuine with respect to $\sigma_1$ and $\sigma_2$.
By Lemma \ref{cn0},
\[
  \pi\in \CR_{\sigma_1}\quad \textrm{or}\quad \pi\in \CR_{\sigma_2}.
\]
This proves the proposition.
\end{proof}

The upper bound in Theorem \ref{main} is an easy consequence of Proposition \ref{dict0}:

\begin{corl}\label{upper}
Let $\pi\in \Irr(\oGb(U))$. Let $\mathbf t_1, \mathbf t_2\in \CW_U$
be two elements so that $\mathbf t_1-\mathbf t_2=\mathbf t_U^\circ$.
Assume that $\pi$ is genuine with respect to $\mathbf t_1$ (and
hence genuine with respect to $\mathbf t_2$). Then
\[
   \operatorname n_{\mathbf t_1}(\pi)+\operatorname n_{\mathbf t_2}(\pi)\leq 2\dim U+\mathrm d_{\rD,\epsilon}.
\]
\end{corl}

\begin{proof}

Write $\tilde{\mathbf t}_i\subset \widehat \CW_U$ for the inverse image of $\mathbf t_i$ under the natural homomorphism $\widehat \CW_U\rightarrow \CW_U$ ($i=1,2$). Then
\[
  \operatorname n_{\mathbf t_i}(\pi)=\min \{\dim \sigma_i\mid \sigma_i\in \tilde{\mathbf t}_i, \, \pi\in \CR_{\sigma_i}\}.
\]
Assume that
\[
   \operatorname n_{\mathbf t_1}(\pi)+\operatorname n_{\mathbf t_2}(\pi)>2\dim U+\mathrm d_{\rD,\epsilon}.
\]
Then Proposition \ref{conserv0} implies that
\[
   \operatorname n_{\mathbf t_1}(\pi)+\operatorname n_{\mathbf t_2}(\pi)= 2\dim U+\mathrm d_{\rD,\epsilon}+2k
\]
for some integer $k>0$.
Therefore there exist $\sigma_1\in \tilde{\mathbf t}_1$ and $\sigma_2\in\tilde{\mathbf t}_2$ so that
\[
  \left\{
    \begin{array}{l}
   \pi\notin \CR_{\sigma_1} \,\textrm{ and }\,   \pi\notin \CR_{\sigma_2};\,\textrm{ and} \\
     \dim \sigma_1+\dim \sigma_2= 2\dim U+\mathrm d_{\rD,\epsilon}-2.\\
        \end{array}
  \right.
  \]
This contradicts Proposition \ref{dict0}.
\end{proof}

\noindent {\bf Remarks}: (a) Let $\sigma_1$ and $\sigma_2$ be as in
Proposition \ref{dict0}. Proposition \ref{conserv0} implies that
$\sigma_1$ is in the stable range (that is, $\rank \,\sigma_1\geq
\dim U$) if and only if $\sigma_2\notin \widehat \CW_U^+$. Therefore
Proposition \ref{dict0} implies that \cite[Propositions 4.3 and
4.5]{Ku3}
\be \label{stablerange}
  \CR_{\sigma}=\{\pi\in \Irr(\oGb(U))\mid \pi \textrm{ is genuine with respect to $\sigma$}\}
\ee
for all $\sigma\in \widehat \CW_U^+$  in the stable range.

(b) It is easy to see that Theorem \ref{main} (the conservation
relations) is equivalent to the following: for all $\sigma_1,
\sigma_2\in \widehat \CW_U$ as in Proposition \ref{dict0}, we have
that
\[
  \CR_{\sigma_1} \sqcup \CR_{\sigma_2}=\{\pi\in \Irr(\oGb(U))\mid \pi \textrm{ is genuine with respect to $\sigma_1$ (and $\sigma_2$)}\}.
\]
For $\dim\, \sigma_1=\dim \, \sigma_2$, the above assertion is
called theta dichotomy in the literature \cite{KR3, HKS}. The theta
dichotomy was established by Harris \cite[Theorem 2.1.7]{Ha} (for
unitary-unitary dual pairs), and by Zorn \cite[Theorem 1.1]{Zo} and
Gan-Gross-Prasad \cite[Theorem 11.1]{GGP} (for orthogonal-symplectic
dual pairs). For a related work of Prasad, see \cite{Pra}.

\section{Non-occurrence of the trivial representation before stable range}
\label{NOtrivial}
Let $U$ be an $\epsilon$-Hermitian right $\rD$-vector space. Recall that $\mathbf t_U^\circ\in \CW_U$ denotes the anti-split Witt tower (Section \ref{kk}).
The main purpose of this section is to show the following proposition,
which is the second key point of this article and which is responsible for the lower bound in conservation
relations.

\begin{prpl}\label{trivialv-ind} One has that
\[
\operatorname n_{\mathbf t_U^\circ}(1_U)\geq 2\dim U+\mathrm d_{\rD,\epsilon}.
\]
Here and as before $1_U\in \Irr(\oGb(U))$ stands for the trivial representation.
\end{prpl}

Proposition \ref{trivialv-ind} is proved in \cite[Appendix]{Ra1}, \cite[Lemma 4.2]{KR3}
and \cite[Theorem 2.9]{GG}, respectively for orthogonal groups, symplectic groups, and unitary groups. Only the quaternionic
case is new. Because of the lack of MVW-involutions, the
approach of \cite{KR3} and \cite{GG}, which uses the doubling method, does
not work for this case. We will follow the idea
of Rallis (\cite{Ra1, Ra2}, which treat the case of orthogonal
groups) to provide a uniform proof of Proposition \ref{trivialv-ind}.

By the argument of Section \ref{secstra}, Proposition \ref{trivialv-ind} implies the following
\begin{prpl}\label{lower0}
Let $\pi\in \Irr(\oGb(U))$. Let $\mathbf t_1, \mathbf t_2\in \CW_U$ be two elements so
that $\mathbf t_1-\mathbf t_2=\mathbf t_U^\circ$. Assume that $\pi$ is genuine with respect to $\mathbf t_1$ (and hence genuine with respect to $\mathbf
t_2$). Then
\[
   \operatorname n_{\mathbf t_1}(\pi)+\operatorname n_{-\mathbf t_2}(\pi^\vee)\geq 2\dim U+\mathrm d_{\rD,\epsilon}.
\]
\end{prpl}

As demonstrated in Section \ref{secstra}, Corollary \ref{upper} and Proposition \ref{lower0} then imply Theorem \ref{main}.

\vsp

Proposition \ref{trivialv-ind} is clear when $U=0$. For the rest of this section, assume that $U\neq 0$, and put $d:=\dim U>0$.
Write $\sigma_U^\circ=(V^\circ, \omega_U^\circ)$ for the anisotropic element of $\mathbf
t_U^\circ$. We view $\omega_U^\circ$ as a representation of
$\oG(U)$ (when $U$ is a symplectic space, the restriction of
$\omega_U^\circ$ to $\oGb(U)$ descends to a representation of $\oG(U)$).
Likewise, view $1_U$ as the trivial representation of $\oG(U)$. Then Proposition \ref{trivialv-ind} amounts to the following
\begin{prpl} \label{vahom} One has that
\begin{equation}
  \Hom_{\oG(U)}(\CS(U^{d-1})\otimes \omega_U^\circ,
  1_U)=0,
\end{equation}
where $U^{d-1}$ carries the diagonal action of $\oG(U)$, and $\CS(U^{d-1})$ carries the induced action of $\oG(U)$.
\end{prpl}

\subsection{Non-occurrence of $1_U$ in $\omega_U^\circ$}
\label{Sub5.2} Let $V$ be a $-\epsilon$-Hermitian left $\rD$-vector
space. We start with the following observation, which is easily seen
from the Schr\"{o}dinger model of an oscillator representation. See
\cite{Li2}, and for a fuller treatment see \cite[Part II, Section
3]{Ho2}.

\begin{lem}\label{nx}
 Let $\omega$ be a smooth oscillator representation of $\oJ_U(V)$. Assume that $U$ is split with a Lagrangian subspace $X$, and $V$ is anisotropic. Then every linear
functional on $\omega$ is $\operatorname N(X)$-invariant if and only
if it is invariant under $X\otimes_\rD V\subset \oH(U\otimes_\rD V)$, where
$\operatorname N(X)$ denotes the unipotent radical of $\oP(X)$.
\end{lem}

Consequently, we have
\begin{lem}\label{splitvhom}
Let $\omega$ be a smooth oscillator representation of
$\oJ_U(V)$. Assume that $U$ is split and non-zero, and
$V$ is anisotropic and non-zero. Then
\begin{equation*}
  \Hom_{\oG(U)}(\omega,  1_U)=0.
\end{equation*}
\end{lem}
\begin{proof}
Let $X$ be a Lagrangian subspace of $U$. Assume that there is a non-zero
element $\lambda \in \Hom_{\oG(U)}(\omega,  1_U)$. Then Lemma
\ref{nx} implies that $\lambda$ is a scalar multiple of
$\lambda_{X\otimes_\rD V}$. This contradicts the equality
\eqref{lambdax}, as all Kudla characters are unitary.

\end{proof}

\begin{lem}\label{dim1}
If  $d=1$, then
\begin{equation}\label{1v}
  \Hom_{\oG(U)}(\omega_U^\circ,  1_U)=0.
\end{equation}
\end{lem}
\begin{proof}
Note that $d=1$ implies that $U$ is not a symplectic space, and $\oGb(U)=\oG(U)$ is a compact group.
Introduce the space $U^\square:=U\oplus U^-$ and its Lagrangian subspace $U^\triangle$ as in Section \ref{dm}. Write $\sigma^\square$ for the anisotropic element of the anti-split Witt tower $\mathbf t_{U^\square}^\circ$.  By Lemma \ref{cn1}, it suffices to show that
\be \label{homuu}
 \Hom_{\oG(U)\times \oG(U^-)}(\RQ_{\sigma^\square}, 1_U \otimes  1_{U^-})=0.
\ee
Note that \eqref{homuu} is implied by the following:
\be \label{homhaa}
 \Hom_{\oG(U)}(\RQ_{\sigma^\square}, 1_U)=0.
\ee

As a simple instance of Proposition \ref{quotient0}, we have an
exact sequence of representations of $\oG(U^\square)$:
\begin{equation}\label{exrho1}
  0\rightarrow \RQ_{\sigma^\square}\rightarrow \operatorname I(\chi_{\sigma^\square, U^\triangle})\rightarrow 1_{U^\square}\rightarrow 0.
\end{equation}
Note that
\[
  (\operatorname I(\chi_{\sigma^\square, U^\triangle}))|_{\oG(U)}\cong \con^\infty(\oG(U)),
\]
where $\oG(U)$ acts on $\con^\infty(\oG(U))$ by right translations. By \eqref{exrho1}, we have a decomposition (recall that $\oG(U)$ is compact)
\[
  \con^\infty(\oG(U))\cong (\RQ_{\sigma^\square})|_{\oG(U)}\oplus 1_U,
\]
and hence \eqref{homhaa} follows by the uniqueness of Haar measure.
\end{proof}

\begin{lem}\label{vtrivial1}
One has that
\begin{equation}\label{1v}
  \Hom_{\oG(U)}(\omega_U^\circ,  1_U)=0.
\end{equation}
\end{lem}
\begin{proof}
When $U$ is a symplectic space, this is a special case of Lemma
\ref{splitvhom}. Now assume that $U$ is not a symplectic space. Then
there is an orthogonal decomposition $U=U_1\oplus U_2$ of $\epsilon$-Hermitian space
such that $\dim U_1=1$. Therefore
\[
  \Hom_{\oG(U)}(\omega_U^\circ,  1_U)\subset \Hom_{\oG(U_1)}(\omega_U^\circ,
  1_{U_1})=\Hom_{\oG(U_1)}(\omega_{U_1}^\circ\otimes \omega_{U_2}^\circ,
  1_{U_1})=0,
\]
by Lemma \ref{dim1}.
\end{proof}

\subsection{Vanishing on small orbits in the null cone}
\label{Sub5.4}
Put
\[
  \Gamma:=\{(v_1,v_2,\cdots,v_{d-1})\in U^{d-1}\mid \la
  v_i,v_j\ra_U=0,\, \,i,j=1,2,\cdots, d-1\},
\]
which is referred to as the null cone in $U^{d-1}$. Write
\[
   \Gamma:=\bigsqcup_{i=0}^{\rank\,U} \Gamma_i,
\]
where
\[
\Gamma_i:=\{(v_1,v_2,\cdots, v_{d-1})\in \Gamma\mid v_1,v_2,\cdots, v_{d-1}\textrm{ span a subspace of $U$ of dimension $i$}\}.
\]
Then each $\Gamma_i$ is a single $\oG(U)\times \GL_{d-1}(\rD)$-orbit. Here $\GL_{d-1}(\rD)$ acts on $U^{d-1}$ by
\be\label{guglact}
  h\cdot(v_1,v_2,\cdots,v_{d-1}):=(v_1,v_2,\cdots,v_{d-1})h^{-1}.
\ee

\begin{lem}\label{vsmall}
If $2i<d$, then
\[
 \Hom_{\oG(U)}(\CS(\Gamma_i)\otimes \omega_U^\circ, 1_U)=0.
\]
 \end{lem}
\begin{proof}
It suffices to show that for each compact open subgroup $L$ of $\GL_{d-1}(\rD)$,
\[
 \Hom_{\oG(U)\times L}(\CS(\Gamma_i)\otimes \omega_U^\circ, 1_U)=0,
\]
where $\omega_U^\circ$ and $1_U$ are viewed as representations of $\oG(U)\times L$ so that $L$ acts trivially. Let $O_i$ be a $\oG(U)\times L$-orbit in $\Gamma_i$ (which is open). We only need to show that
\[
 \Hom_{\oG(U)\times L}(\CS(O_i)\otimes \omega_U^\circ, 1_U)=0.
\]
By Frobenius reciprocity \cite[Chapter I, Proposition 2.29]{BZ}, one has that
\[
 \Hom_{\oG(U)\times L}(\CS(O_i)\otimes \omega_U^\circ, 1_U)=\Hom_{(\oG(U)\times L)_{\mathbf v}}(\delta_{\mathbf v} \otimes \omega_U^\circ, 1_U),
\]
where
\[
  \mathbf v=(v_1,v_2,\cdots, v_{d-1})\in O_i,
\]
$(\oG(U)\times L)_{\mathbf v}$ is the stabilizer of $\mathbf v$ in $\oG(U)\times L$, and $\delta_{\mathbf v}$ is a certain positive character on $(\oG(U)\times L)_{\mathbf v}$.

Since $2i<d$, there is a non-zero non-degenerate subspace $U_0$ of $U$ which is perpendicular to $v_1, v_2, \cdots, v_{d-1}$. Note that $\oG(U_0)\subset (\oG(U)\times L)_{\mathbf v}$ and $\delta_{\mathbf v}$ is trivial on  $\oG(U_0)$. Therefore,
\begin{eqnarray*}
   & & \Hom_{(\oG(U)\times L)_{\mathbf v}}(\delta_{\mathbf v} \otimes \omega_U^\circ, 1_U) \\
   &\subset &\Hom_{\oG(U_0)}(\omega_U^\circ, 1_U)\\
   &=&\Hom_{\oG(U_0)}(\omega_{U_0}^\circ\otimes \omega_{U_0^\perp}^\circ, 1_{U_0})\\
   &=&0 \qquad(\textrm{by Lemma \ref{vtrivial1}}).
\end{eqnarray*}
\end{proof}

\subsection{A homogeneity calculation for the main orbits in the null cone}
\label{Sub5.5}

Let $\rE^\times$  act on $U^{d-1}$ by
\[
  a\cdot x:=xa^{-1}, \quad a\in\rE^\times, \,x\in U^{d-1}.
\]
Denote by $\mathcal O_\rD$ the ring of integers in $\rD$:
\[
  \mathcal O_\rD:=\{a\in \rD\mid \abs{\det a}_\rE\leq 1\}.
\]
For each $i=0,1,\cdots, \rank \,U$, using the decomposition
\[
  \GL_{d-1}(\rD)=\GL_{d-1}(\mathcal O_\rD)
   \left\{
        \left[
               \begin{array}{cc}
                 g & 0 \\
                 * & h \\
               \end{array}
             \right]
              \in \GL_{d-1}(\rD)\mid g\in \GL_i(\rD), \,h\in \GL_{d-1-i}(\rD)
       \right\},
\]
it is easy to see that $\Gamma_i$ is a homogeneous space for the action of  $\oG(U)\times \GL_{d-1}(\mathcal O_\rD)$. Consequently, $\Gamma_i$ is a homogeneous space for the action of  $\oG(U)\times \GL_{d-1}(\mathcal O_\rD)\times \rE^\times$.

In the rest of this subsection, assume that $U$ is split, and write $d=2r>0$. Put
\[
  \rho_r:=\frac{2r+\mathrm d_{\rD,\epsilon}-2}{4},
\]
as in \eqref{defrhor}. Recall that $\delta_\rD$ (the degree of $\rD$
over $\rE$) equals $2$ if $\rD$ is a quaternion algebra, and equals
$1$ otherwise. The following lemma is an easy consequence of
\cite[Theorem 33D]{Lo}. We omit the details.
\begin{lem}\label{transo}
Up to scalar multiple, there exists a unique positive Borel measure $\mu_{\Gamma_r}$ on $\Gamma_r$ such that
\[
  (g,h,a)\cdot \mu_{\Gamma_r}=\abs{a}_\rE^{2 \delta_\rD^{2}\, r\rho_r}\mu_{\Gamma_r}, \qquad  (g,h,a)\in \oG(U)\times \GL_{d-1}(\mathcal O_\rD)\times \rE^\times,
\]
where $(g,h,a)\cdot \mu_{\Gamma_r}$ denotes the push-forward of $\mu_{\Gamma_r}$ through the action of $(g,h,a)$ on $\Gamma_r$.
\end{lem}

We will use the following convention for the rest of this section: given a group $G$ acting on two sets $A$ and $B$, then for every $g\in G$ and every map $\varphi: A\rightarrow B$, $g\cdot \varphi: A\rightarrow B$ is the map defined by
\[
  (g\cdot \varphi)(a):=g\cdot(\varphi(g^{-1}\cdot a)), \quad a\in A.
\]
If no action of $G$ is specified on a set $C$, we consider $C$ to carry the trivial action of $G$.

Note that each $\oG(U)$-orbit in $\Gamma$ is $\RE^\times$-stable. We
shall examine $\oG(U)$-orbits in $\Gamma_r$.

\begin{lem}\label{homd0}
Let $O_r$ be a $\oG(U)$-orbit in $\Gamma_r$. Then the space $\Hom_{\oG(U)}(\omega_U^\circ, \con^{\infty}(O_r))$ is one dimensional and every element $\lambda$ of the space satisfies
\[
  a \cdot \lambda=\abs{a}_\rE^{-2 \delta_\rD^{2}\, r\rho_1}\lambda, \quad a\in \rE^\times.\qquad(\,\rho_1=\frac{\mathrm d_{\rD,\epsilon}}{4}\,)
\]
\end{lem}

\begin{proof}
Fix an element $\mathbf v=(v_1,v_2,\cdots,v_{d-1})\in O_r$. Denote by $X$ the Lagrangian subspace of $U$ spanned by $v_1,v_2,\cdots, v_{d-1}$. Fix a Lagrangian subspace $Y$ of $U$ which is complementary to $X$. For every $a\in \rE^\times$, denote by $m_a\in \oG(U)$ the element which stabilizes both $X$ and $Y$, and acts on $X$ through the scalar multiplication by $a$.

The stabilizer of $\mathbf v$ in $\oG(U)$ equals $\operatorname N(X)$, the unipotent radical of the parabolic subgroup $\oP(X)$. Therefore
\[
   O_r=\oG(U)/\operatorname N(X).
\]
The corresponding action of $\rE^\times$ on $\oG(U)/\operatorname  N(X)$ is given by
\[
  a\cdot(g\operatorname N(X))=gm_{a}^{-1}\operatorname N(X), \quad a\in E^\times.
\]

By Frobenius reciprocity,
\begin{equation}\label{fro}
  \Hom_{\oG(U)}(\omega_U^\circ, \con^{\infty}(O_Z))=\Hom_{\operatorname N(X)}(\omega_U^\circ, \C).
\end{equation}
It is easy to check that, under the identification \eqref{fro}, the action of $\rE^\times$ on the left hand side corresponds to the following action on the right hand side:
\[
  a\cdot \lambda :=\lambda \circ (\omega_U^\circ(m_a^{-1})), \quad a\in \rE^\times, \, \lambda \in \Hom_{\operatorname N(X)}(\omega_U^\circ, \C).
\]

By Lemma \ref{nx}, $\Hom_{\operatorname N(X)}(\omega_U^\circ, \C)$ is spanned by $\lambda_{X\otimes V^\circ}$, and \eqref{lambdax} implies that
\[
  \lambda_{X\otimes V^\circ}\circ
  (\omega_U^\circ(m_a^{-1}))=\abs{a}_\rE^{-2 \delta_\rD^{2}\, r\rho_1}\lambda_{X\otimes V^\circ},\qquad a\in \RE^\times.
  \]
  This proves the lemma.
\end{proof}

\begin{lem}\label{homdl}
Every element $\lambda\in \Hom_{\oG(U)}(\CS(\Gamma_r)\otimes\omega_U^\circ, 1_U)$ satisfies
\[
   a \cdot \lambda=\abs{a}_\rE^{2r \delta_\rD^2\,(\rho_r-\rho_1)}\lambda, \qquad a\in \rE^\times.
\]
\end{lem}

\begin{proof} Without loss of generality, assume that $\lambda$ is fixed by an open subgroup of $\GL_{d-1}(\mathcal O_\rD)$. Then $\lambda$ naturally corresponds to an element of $\Hom_{\oG(U)}(\omega_U^\circ, \con^\infty(\Gamma_r)\mu_{\Gamma_r})$, where $\mu_{\Gamma_r}$ is as in Lemma \ref{transo}. Therefore the lemma follows by Lemmas \ref{transo} and \ref{homd0}, and by considering the following product of restriction maps:
\[
  \con^{\infty}(\Gamma_r)\hookrightarrow \prod_{O_r\textrm{ is a $\oG(U)$-orbit in $\Gamma_r$}} \con^\infty (O_r).
\]
\end{proof}

\subsection{The Fourier transform}
For every $\epsilon$-Hermitian right $\rD$-vector space $U'$, define the Fourier transform
\[
  \CF_{U'}: \CS(U')\rightarrow \CS(U')
\]
by
\[
  (\CF_{U'}(\phi))(x):=\int_{U'} \phi(y)\,\psi\left(\frac{\la x,y\ra_{U'}+\la x,y\ra_{U'}^\iota}{2}\right)\,dy,\quad \phi\in \CS(U'),\,x\in U',
\]
where $dy$ is a fixed Haar measure on $U'$. It is easy to check that
\begin{equation}\label{invg}
   \mathcal F_{U'}(g\cdot \phi)=g\cdot(\mathcal F_{U'}(\phi)), \quad g\in \oG(U'), \,\phi\in \CS(U'),
\end{equation}
and
\begin{equation}\label{inva}
    \mathcal F_{U'}(a\cdot \phi)=\abs{a}_\rE^{-\dim_\rE U'} (a^{-1})^\iota\cdot (\mathcal F_{U'}(\phi)), \quad a\in \rE^\times, \,\phi\in \CS(U'),
\end{equation}
where the action of $\rE^\times$ on $\CS(U')$ is given by
\[
  (a\cdot\phi)(x):=\phi(xa),\quad a\in \rE^\times, \,\phi\in \CS(U'),\,x\in U'.
\]

We refer the reader to the notation of the last subsection. For every linear functional $\lambda$ on $\CS(U^{d-1})\otimes\omega_U^\circ$, define its Fourier transform to be the linear functional
\[
  \widehat \lambda: \CS(U^{d-1})\otimes\omega_U^\circ\rightarrow \C,\quad \phi\otimes \phi'\mapsto \lambda(\CF_{U^{d-1}}(\phi)\otimes \phi').
\]
Using extension by zero, we get an inclusion
\[
   \CS(U^{d-1}\setminus \Gamma)\subset \CS(U^{d-1}).
\]
(Recall that $\Gamma $ is the null cone in $U^{d-1}$.)

\begin{lem}\label{homf}
Let $\lambda\in \Hom_{\oG(U)}(\CS(U^{d-1})\otimes\omega_U^\circ, 1_U)$. If both $\lambda$ and $\widehat \lambda$ vanish on the subspace $\CS(U^{d-1}\setminus \Gamma)\otimes\omega_U^\circ$, then $\lambda=0$.
\end{lem}
\begin{proof}
If $U$ is not split, then the lemma follows from Lemma \ref{vsmall}. Now assume that $U$ is split. Then Lemma \ref{vsmall}  and Lemma \ref{homdl} imply that
\be \label{homgla}
  a \cdot \lambda=\abs{a}_\rE^{2r \delta_\rD^2\,(\rho_r-\rho_1)}\lambda, \qquad a\in \rE^\times.
\ee
It is easily checked that \eqref{homgla} and \eqref{inva} imply that
\be \label{homgla3}
  a \cdot \widehat \lambda=\abs{a}_\rE^{\delta_\rD^2 \,d(d-1)-2r \delta_\rD^2\,(\rho_r-\rho_1)}\widehat \lambda, \qquad a\in \rE^\times.
\ee
Note that \eqref{invg} implies that $\widehat \lambda\in \Hom_{\oG(U)}(\CS(U^{d-1})\otimes\omega_U^\circ, 1_U)$. Similarly to \eqref{homgla}, we have that
\be \label{homgla4}
  a \cdot \widehat \lambda=\abs{a}_\rE^{2r \delta_\rD^2\,(\rho_r-\rho_1)}\widehat \lambda, \qquad a\in \rE^\times.
\ee
Since $d=2r$ and $\rho_r-\rho_1= \frac{r-1}{2}$, we will have
\[
  \delta_\rD^2 \,d(d-1)-2r \delta_\rD^2\,(\rho_r-\rho_1)\neq 2r \delta_\rD^2\,(\rho_r-\rho_1),
\]
and so we conclude that $\lambda=0$ by comparing \eqref{homgla3} and \eqref{homgla4}.
\end{proof}

\begin{corl}\label{homfcor}
If $U$ is a symplectic space of dimension $2$, then $\Hom_{\oG(U)}(\CS(U)\otimes\omega_U^\circ, 1_U)=0$.
\end{corl}
\begin{proof}
In this case, $\Gamma=U$. Therefore this is a special case of Lemma \ref{homf}.
\end{proof}

\subsection{Reduction to the null cone and conclusion of the proof}
\label{Sub5.3}

By Lemma \ref{dim1} and Corollary \ref{homfcor}, Proposition \ref{vahom} holds when $d=1$, or when $U$ is  a symplectic space and $d=2$. We prove Proposition \ref{vahom} by induction on $d$. So assume that $d\geq 4$ when $U$ is a symplectic space, and $d\geq 2$ in all other cases, and assume that Proposition \ref{vahom} holds when $d$ is smaller.

Let $U_0$ be a non-zero non-degenerate subspace of $U$ of dimension $d_0$, where $d_0=2$ if $U$ is a symplectic space, and $d_0=1$ otherwise. Denote by $U_0^\perp$ the orthogonal complement of $U_0$ in $U$. Put
\[
  B_0:=\left\{
         \begin{array}{ll}
           \{(v_1,v_2)\in (U_0)^2\mid \textrm{$v_1, v_2$ is a basis of $U_0$}\}, & \hbox{if $U$ is a symplectic space;} \\
           U_0\setminus \{0\}, & \hbox{otherwise.}
         \end{array}
       \right.
\]
Then
\[
  S_0:=B_0\times U^{d-d_0-1}\subset U^{d-1}
\]
is stable under $\oG(U_0^\perp)\subset \oG(U)$, and the map
\begin{equation}\label{subm1}
    \oG(U)\times S_0\rightarrow U^{d-1}, \quad (g, \mathbf v)\mapsto g \cdot \mathbf v
\end{equation}
is  $\oG(U)\times \oG(U_0^\perp)$-equivariant, where $\oG(U)\times \oG(U_0^\perp)$ acts on $\oG(U)\times S_0$ by
 \[
   (g,h)\cdot (x, \mathbf v):=(gxh^{-1}, h \cdot \mathbf v),
 \]
 and acts on $U^{d-1}$ by
 \[
   (g,h)\cdot \mathbf v:=g\cdot \mathbf v.
 \]

 \begin{lem}\label{vhompl}
 One has that
 \begin{equation}\label{vhomp}
   \Hom_{\oG(U)\times \oG(U_0^\perp)}(\CS(\oG(U)\times S_0)\otimes \omega_U^\circ, 1_U)=0,
   \end{equation}
   where the representations $\omega_U^\circ$ and $1_U$ of $\oG(U)$  are extended to the group $\oG(U)\times \oG(U_0^\perp)$ by the trivial action of $\oG(U_0^\perp)$.
  \end{lem}
 \begin{proof}
 Frobenius reciprocity \cite[Chapter I, Proposition 2.29]{BZ} implies that the left hand side of \eqref{vhomp} equals
  \begin{equation}\label{frob1}
   \Hom_{\oG(U_0^\perp)}(\CS(S_0)\otimes (\omega_U^\circ)|_{\oG(U_0^\perp)}, 1_{U_0^\perp}).
  \end{equation}
 Note that
 \[
   \CS(S_0)=\CS(B_0\times U_0^{d-d_0-1})\otimes \CS((U_0^\perp)^{d-d_0-1}),
 \]
 and
 \[
   \omega_U^\circ=\omega_{U_0}^\circ\otimes \omega_{U_0^\perp}^\circ.
 \]
 By the induction assumption, we have
 \[
 \Hom_{\oG(U_0^\perp)}(\CS((U_0^\perp)^{d-d_0-1})\otimes
   \omega_{U_0^\perp}^\circ , 1_{U_0^\perp})=0,
 \]
 and therefore the space \eqref{frob1} vanishes as well.
 \end{proof}

\begin{lem}\label{reducn}
One has that
\[
  \Hom_{\oG(U)}(\CS(U^{d-1}\setminus \Gamma)\otimes
   \omega_U^\circ , 1_U)=0.
\]
\end{lem}
\begin{proof}
Note that $\oG(U)\cdot S_0$ is open in $U^{d-1}$, and the push-forward of measures through the map \eqref{subm1} induces a $\oG(U)\times \oG(U_0^\perp)$-equivariant surjective linear map
\[
  \CS(\oG(U)\times S_0)\,(\mu_{\oG(U)}\otimes \mu_{S_0}) \rightarrow \CS(\oG(U)\cdot S_0)\, \mu_{\oG(U)\cdot S_0},
\]
where $\mu_{\oG(U)}$ is a Haar measure on $\oG(U)$, $\mu_{S_0}$ is the restriction of a Haar measure on $U_0^{d_0}\times U^{d-d_0-1}$ to $S_0$, and $\mu_{\oG(U)\cdot S_0}$ is the restriction of a Haar measure on $U^{d-1}$ to $\oG(U)\cdot S_0$. (This is because that the map \eqref{subm1} is a submersion between locally analytic manifolds over $\rF$ \cite{Sc}.) Consequently, there exists  a $\oG(U)\times \oG(U_0^\perp)$-equivariant surjective linear map
\[
  \CS(\oG(U)\times S_0) \rightarrow \CS(\oG(U)\cdot S_0).
\]
Therefore Lemma \ref{vhompl} implies that
\[
 \Hom_{\oG(U)}(\CS(\oG(U)\cdot S_0)\otimes \omega_U^\circ, 1_U)=\Hom_{\oG(U)\times
 \oG(U_0^\perp)}(\CS(\oG(U)\cdot S_0)\otimes \omega_U^\circ, 1_U)=0.
\]
This further implies that
\[
  \Hom_{\oG(U)}(\CS(\oG(U)\cdot (h\cdot S_0))\otimes \omega_U^\circ, 1_U)=0.
\]
for all $h\in \GL_{d-1}(\rD)$.  The lemma then follows by noting
that
\[
  \bigcup_{U_0,\,h} \oG(U)\cdot (h\cdot S_0)=U^{d-1}\setminus \Gamma,
\]
where $U_0$ runs through all non-degenerate subspaces of $U$ of
dimension $d_0$, and $h$ runs through all elements of
$\GL_{d-1}(\rD)$.
\end{proof}

Finally, Proposition \ref{vahom} follows by combining Lemma \ref{homf} and  Lemma \ref{reducn}.






\section{The archimedean case}\label{secarch}

\subsection{The generalized Witt-Grothendieck groups}

In the non-archimedean case, we work with the class of smooth representations of totally disconnected locally compact topological groups. For the archimedean case, we shall replace this by moderate growth smooth Fr\'{e}chet representations of almost linear Nash groups. Recall that a Nash group is said to be almost linear if it has a Nash representation with finite kernel. See \cite{Su2} for details on almost linear Nash groups. For the definition of moderate growth smooth Fr\'{e}chet representations of almost linear Nash groups, see \cite[Definition 1.4.1]{du} or \cite[Section 2]{Su3}.

Let $(\rF, \rD, \epsilon)$, $U$ and $V$ be as in Section \ref{sub1.1}. Recall that $\psi: \rF\rightarrow \C^\times$ is a fixed non-trivial unitary character. In this section, assume that $\rF$ is archimedean. Then the groups $\oGb(U)$, $\oGb(V)$, $\oH(U\otimes_\rD V)$, $\oJb_U(V)$ and $\oJb(U,V)$  are all naturally almost linear Nash groups. Denote by $\Irr(\oGb(U))$ the set of all isomorphism classes of irreducible Casselman-Wallach representations of $\oGb(U)$. Recall that a moderate growth smooth Fr\'{e}chet representation of $\oGb(U)$ is called a Casselman-Wallach representation if its Harish-Chandra module has finite length.  The reader may consult \cite{Cass} and \cite[Chapter 11]{Wa2} for more information about Casselman-Wallach
representations.

Replacing smooth representations in the non-archimedean case by  moderate growth smooth Fr\'{e}chet representations, we define in the archimedean case the notion of smooth oscillator representations as in Definition \ref{defso}, and then enhanced oscillator representations of $\oGb(U)$ as in Definition \ref{defenh}. The monoid $\widehat \CW_U^+$, the groups $\widehat \CW_U$ and $\CW_U$, and the inverse limits $\widehat \CW_\infty^+$, $\widehat \CW_\infty$ and $\CW_\infty$, are defined exactly as in the non-archimedean case.

Also when $U$ is split and non-zero, define the group $\Hom( \oGb(U)_{\mathrm{split}}, \C^\times)$ as in Section \ref{k1}. Then there is a natural isomorphism
\[
  \Hom( \oGb(U)_{\mathrm{split}}, \C^\times)\cong \RK^*,
\]
where
\[
  \RK:=\left\{
                  \begin{array}{ll}
                   \R^\times/\R^\times_+, &\hbox{if $U$ is a real symmetric bilinear space};\\
                   \{1\}, &\hbox{if $U$ is a complex symmetric bilinear space};\\
                  \oHil(\R)\cong \Z/4\Z, &\hbox{if $U$ is a real symplectic space};\\
                   \{\pm 1\},&\hbox{if $U$ is a complex symplectic space};\\
                   \RE^\times/\R^\times_+, & \hbox{if $\rD$ is a quadratic extension;}\\
                  \{1\},&\hbox{if $\rD$ is a  quaternion algebra.}
                  \end{array}
                \right.
\]
Here and as before, $\RE$ denotes the center of $\rD$,  and $\oHil(\R)$ is defined as in \eqref{atat}.

Similar to Theorem \ref{cwinf}, for $\rF=\R$, the group $\widehat \CW_\infty$ is canonically isomorphic to the group of Table \ref{tablecwir}.

\begin{table}[h]
\caption{The group $\widehat \CW_\infty$ (for $\rF=\R$)}\label{tablecwir}
\centering 
\begin{tabular}{c c c c c c c} 
\hline
$\rD$ & \vline & $\rF$ & quadratic extension & quaternion algebra\\
\hline
 $\epsilon=1$ & \vline & $2\Z \times \left(\frac{\R^\times}{\R_+^\times}\right)^*\times \{\pm 1\}^*$ & $\left(\bigoplus_{\varpi\in \frac{\rE_-^\times}{\R^\times_+}}\Z\varpi\right)\times_{{\Z}/{2\Z}}\left(\frac{\rE^\times}{\R_+^\times}\right)^*$  & $\phantom{\frac{\overbrace a}{\underbrace a}}\Z\phantom{\frac{\overbrace a}{a}}$\\ 
\hline
$\epsilon=-1$ & \vline &$\bigoplus_{\varpi\in \frac{\R^\times}{\R^\times_+}}\Z\varpi$ & $\left(\bigoplus_{\varpi\in \frac{\R^\times}{\R^\times_+}}\Z\varpi\right)\times_{{\Z}/{2\Z}} \left(\frac{\rE^\times}{\R_+^\times}\right)^*$  & $\phantom{\frac{\overbrace a}{\underbrace a}}\bigoplus_{\varpi\in \frac{\R^\times}{\R^\times_+}}\Z\varpi \phantom{\frac{\overbrace a}{a}}$  \\
\hline  
\end{tabular}
\label{table:nonlin} 
\end{table}

In Table \ref{tablecwir}, $\rE_-^\times:=\{a\in \rE^\times\mid a^\iota=-a\}$; and for the data in the definitions of the fiber products, we are given the homomorphisms $\bigoplus_{\varpi\in \rE_-^\times/\R^\times_+}\Z\varpi\rightarrow \Z/2\Z$ and  $\bigoplus_{\varpi\in \R^\times/\R^\times_+}\Z\varpi\rightarrow \Z/2\Z$ which map the free generators to $1+2\Z$, and the homomorphism $\left(\frac{\rE^\times}{\R_+^\times}\right)^*\rightarrow \Z/2\Z$  whose kernel equals $\left(\frac{\rE^\times}{\R^\times}\right)^*\subset \left(\frac{\rE^\times}{\R^\times_+}\right)^*$.

Identify $\widehat \CW_\infty$ with the group of  Table \ref{tablecwir}. Then the homomorphism $\widehat \xi_\infty: \widehat \CW_\infty\rightarrow \rK^*$ (as in \eqref{kcw}) is identical to the obvious projection map  except for the case when  $\epsilon=-1$ and $\rD=\rF$. In this exceptional case, the homomorphism
\be\label{homspr}
\widehat \xi_\infty: \widehat \CW_\infty=\bigoplus_{\varpi\in \R^\times/\R^\times_+}\Z\varpi\rightarrow \rK^*
\ee
maps the free generator $\varpi_a:=a \R^\times_+$ ($a\in \R^\times$) to $\gamma_{\psi_a}$, where $\psi_a$ denotes the character
\[
\rF=\R\rightarrow \C^\times,\quad x\mapsto \psi(ax),
\]
and $\gamma_{\psi_a}$ is the character on $\oHil(\R)$ defined as in \eqref{gammap}. By the explicit calculation of the Weil indices of real quadratic spaces \cite[Section 26]{Weil}, we know that the kernel of the homomorphism \eqref{homspr} equals
\be\label{kersyr}
  \{a \varpi_1+b\varpi_{-1}\mid a,b\in \Z,\,a-b\in 4\Z\}.
\ee

For $\rF\cong \C$ (then $\rD =\rF$), it is easy to see that there is a canonical isomorphism:
\begin{equation*}
  \widehat \CW_\infty \cong \left\{
  \begin{array}{ll}
  2\Z\times \{\pm 1\}^*, \ \ \ \ &\hbox{if $\epsilon=1$;}\smallskip \\
  \Z, \ \ \ \ &\hbox{if $\epsilon=-1$,} \\
  \end{array}
  \right.
\end{equation*}
and $\widehat \xi_\infty: \widehat \CW_\infty\rightarrow \rK^*$ is the unique surjective homomorphism.

In all cases (for $\rF=\R$ or $\rF \cong \C$), the Kudla homomorphism $\xi_\infty:
\CW_\infty\rightarrow \RK^*$ (as in Section \ref{kk}) is surjective.

\subsection{Conservations relations}
The archimedean analogue of the following basic results remain true: the smooth version of Stone-von Neumann Theorem \cite{du2}, Howe Duality Conjecture \cite{Ho3}, non-vanishing of theta lifting in the stable range \cite{PP}, and Kudla's persistence principle. For each $\mathbf t\in \CW_U$ and for each $\pi\in \Irr(\oGb(U))$ which is genuine (as in Section \ref{firsoc}) with respect to $\mathbf t$, define the first occurrence index $\operatorname n_\mathbf t(\pi)$ as in \eqref{Symp-index}.

On the first occurrences, three different phenomena occur in the archimedean case. As in the non-archimedean case, we will need to use some results on the structure of degenerate principal series of $\oGb(U)$ for $U$ split. We refer the reader to Proposition \ref{quotient0} for the relevant notations.

\vsp
\noindent {\bf Case 1}: $U$ is a real or complex symmetric bilinear space. Then the kernel of the Kudla homomorphism $\xi_\infty: \CW_\infty\rightarrow \RK^*$  has order $2$. Define the anti-split Witt tower  $\mathbf t_U^\circ\in \CW_U$ as in the non-archimedean case. It corresponds to the sign character of the orthogonal group $\oO(U)$. The same results as Proposition \ref{quotient0} and Proposition \ref{trivialv-ind} also hold in this case (see \cite[Section 4]{LZ2}, \cite[Theorem 1]{LZ3}, and \cite[Appendix C]{Pr2}).  Then the argument as in the non-archimedean case shows that the same
conservation relations hold:

\begin{thm}\label{main2}
Let $U$ be a real or complex symmetric bilinear space. Let $\mathbf t_1$ and $\mathbf t_2$ be
two Witt towers in $\CW_U$ with difference $\mathbf t_U^\circ$. Then for every $\pi\in \Irr(\oG(U))$ one has that
\[
   \operatorname n_{\mathbf t_1}(\pi)+\operatorname n_{\mathbf t_2}(\pi)=2\dim U.
\]
\end{thm}

\vsp

\noindent {\bf Case 2}: $U$ is a complex symplectic space or a real quaternionic Hermitian space. Then $\oG(U)$ is a perfect group,  and $\CW_U$ has two elements, namely the Witt tower of even dimensional enhanced oscillator representations, and the Witt tower of odd dimensional enhanced oscillator representations. Concerning degenerate principal series, we have
\begin{prpt}\label{quotient2} (\cite[Theorem 1, Case I]{LZ3} and \cite[Corollary 10.5, (2)]{Ya})
Let $U$ be a complex symplectic space or a real quaternionic Hermitian space. Assume that $U$ is split and let $X$ be a Lagrangian subspace of $U$. Then $\RQ_\sigma=\operatorname I(\chi_{\sigma,X})$ for all $\sigma\in \widehat \CW_U$ such that $\dim \sigma\geq \rank\, U$.
\end{prpt}

It turns out that there is no conservation relation in the case under consideration. Instead, using Proposition \ref{quotient2}, the argument as in Section \ref{dpsdm} yields the following:

\begin{thm}\label{main2}
Let $U$ be a complex symplectic space or a real quaternionic Hermitian space.  Let $\sigma=(V,\omega)\in \widehat \CW_U^+$ and let $\pi\in \Irr(\oGb(U))$. Assume that $\pi$ is genuine with respect to $\sigma$. If $\dim \sigma\geq \dim U$, then
\[
  \Hom_{\oGb(U)}(\omega, \pi)\neq 0.
\]
Consequently, for each $\mathbf t\in\CW_U$ such that $\pi$ is genuine with respect to $\mathbf t$, one has that
\[
   \operatorname n_{\mathbf t}(\pi)\leq\left\{
                                         \begin{array}{ll}
                                           \dim U, & \hbox{if $\dim U\in \dim \mathbf t$;} \\
                                          \dim U+1, & \hbox{otherwise.}
                                         \end{array}
                                       \right.
 \]
Here $\dim \mathbf t\in \Z/2\Z$ denotes the parity of the dimension of an element of $\mathbf t$.
\end{thm}

\vsp
\noindent {\bf Case 3}: $U$ is a real symplectic space, a complex Hermitian or skew-Hermitian space, or a real quaternionic skew-Hermitian space. When $U$ is a complex Hermitian space, let $\varpi_+$ and $\varpi_{-}$ be the two different elements of the set $\rE_-^\times/\R^\times_+$; otherwise,  let $\varpi_+$ and $\varpi_{-}$ be the two different elements of the set $\R^\times/\R^\times_+$.  Identify $\widehat \CW_\infty$ with the group of Table \ref{tablecwir}. Then
\be\label{kerwxi}
  \ker \widehat \xi_\infty=\{\,a \varpi_+ + b\varpi_{-}\mid a,b\in \Z,\,a-b\in \mathrm d_{\rD,\epsilon} \, \Z\,\},
\ee
where $\mathrm d_{\rD, \epsilon}$ is as in \eqref{drhov0}.
Denote by $\H_\infty$ the hyperbolic plane in $\widehat \CW_\infty$, namely the element of $\widehat \CW_\infty$ whose image under the natural homomorphism $\widehat \CW_\infty\rightarrow \widehat \CW_U$ equals $\H_U$, for every $\epsilon$-Hermitian right $\rD$-vector space $U$. Here the hyperbolic plane $\H_U\in \widehat \CW_U$ is defined as in the non-archimedean case. Then under the identification of $\widehat \CW_\infty$ with the group of Table \ref{tablecwir}, we have that $\H_\infty=\varpi_+ + \varpi_{-}$. Therefore, \eqref{kerwxi} implies that
\[
\ker \xi_\infty \cong \mathrm d_{\rD,\epsilon}\,\Z.
 \]
Denote by $\mathcal K_U\subset \CW_U$ the image of $\ker \xi_\infty$ under the natural homomorphism $\CW_\infty\rightarrow \CW_U$, which is also isomorphic to $\mathrm d_{\rD,\epsilon}\,\Z$.

For degenerate principal series representations, we have the following two propositions:

\begin{prpt}\label{quotient4} (\cite[Introduction]{LZ1}, \cite[Section 4]{LZ2}, \cite[Section 6]{LZ4})
Let $U$ be a real symplectic space, or a complex Hermitian or skew-Hermitian space. Assume that $U$ is split and let $X$ be a Lagrangian subspace of $U$. Then for all $\sigma\in \widehat \CW_U$ such that $\dim \sigma \leq \frac{\dim U+\mathrm d_{\rD, \epsilon}-2}{2}$, one has that
\[
  \frac{\operatorname I(\chi'_{\sigma,X})}{\sum_{\sigma'} \RQ_{\sigma'}}\cong \RQ_\sigma
\]
as representations of $\oGb(U)$, where $\sigma'$ in the summation runs through all elements of $\widehat \CW_U$ such that
\be\label{quoarch}
 \left\{
   \begin{array}{l}
      \textrm{$\sigma'-\sigma$ represents a non-zero element of $\CK_U$; and}\smallskip \\
     \dim \sigma'+\dim \sigma=\dim U+\mathrm d_{\rD, \epsilon}-2,
   \end{array}
 \right.
\ee
and $\chi'_{\sigma,X}:=\chi_{\sigma',X}$ for an arbitrary element $\sigma'\in \widehat \CW_U$ satisfying \eqref{quoarch}.
\end{prpt}

\begin{prpt}\label{quotient5} (\cf \cite[Corollary 10.5]{Ya})
Let $U$ be a real quaternionic skew-Hermitian space. Assume that $U$ is split and let $X$ be a Lagrangian subspace of $U$. Then for all integers $m, m'$ such that $m+m'=\dim U-1$ and $m'\geq m$, one has that
\[
  \frac{\operatorname I(\chi_{m',X})}{\sum_{\sigma' \in \widehat \CW_U,\, \dim \sigma'=m'} \RQ_{\sigma'}}\cong \sum_{\sigma \in\widehat \CW_U,\, \dim \sigma=m} \RQ_{\sigma}
\]
as representations of $\oG(U)$, where  $\chi_{m',X}:=\chi_{\sigma',X}$ for an arbitrary element $\sigma'\in \widehat \CW_U$ of dimension $m'$.
\end{prpt}

On the first occurrences, we have the following
\begin{thm}\label{realsp}
 Let $U$ be a real symplectic space,  a complex Hermitian or skew-Hermitian space, or a real quaternionic skew-Hermitian space.  Let $\mathcal T\subset \CW_U$ be a $\mathcal K_U$-coset. Let $\pi\in \Irr(\oGb(U))$ which is genuine with respect to some (and hence all)
elements of $\mathcal T$. Then there are two different elements
$\mathbf t_1, \mathbf t_2\in \mathcal T$ such that
\be\label{consarch}
  \operatorname n_{\mathbf t_1}(\pi)+\operatorname n_{\mathbf
  t_2}(\pi)=2\dim U+\mathrm d_{\rD,\epsilon};
\ee
and for all different elements $\mathbf t_3, \mathbf t_4\in
\mathcal T$, one has that
\be\label{geqn}
   \operatorname n_{\mathbf t_3}(\pi)+\operatorname n_{\mathbf
  t_4}(\pi)\geq 2\dim U+\mathrm d_{\rD,\epsilon}\,\abs{\mathbf t_3-\mathbf
  t_4},
\ee
where for $\mathbf t\in \mathcal K_U$, $\abs{\mathbf t}$ denotes the non-negative integer so
that $\mathbf t$ is $\abs{\mathbf t}$-multiple of a generator of $\mathcal K_U$. Consequently
the following conservation relations hold:
\[
 \sum_{\mathcal Q\in \mathcal T/2\mathcal K_U} \min\{\operatorname n_{\mathbf
 t}(\pi)\mid \mathbf t\in \mathcal Q\}=2\dim U+\mathrm d_{\rD,\epsilon}.
\]
\end{thm}

\begin{proof}
Summarizing the results in \cite[Appendix C]{Pr2}, \cite[Lemma 3.1]{Pa}, \cite[Proposition
3.38]{LPTZ}, and \cite[Theorem 1.2.1]{LL}, we know that the trivial representation $1_U$ does not occur before stable range in every non-split Witt tower in $\CK_U$, that is,
\be\label{nonta}
   \operatorname n_{\mathbf
 t}(1_U)\geq 2\dim U+\mathrm d_{\rD,\epsilon}\,\abs{\mathbf t},\quad \mathbf t\in \CK_U\setminus \{0\}.
\ee
As in Section \ref{secstra}, \eqref{nonta} implies that
\be \label{geqn22}
   \operatorname n_{\mathbf t_3}(\pi)+\operatorname n_{-\mathbf
  t_4}(\pi^\vee)\geq 2\dim U+\mathrm d_{\rD,\epsilon}\,\abs{\mathbf t_3-\mathbf
  t_4},
\ee
for all different elements $\mathbf t_3, \mathbf t_4\in \CT$.

On the other hand, using MVW-involutions on archimedean metaplectic
groups and classical groups (\cf \cite{MVW, Pr1, Su1, LST}), one knows that for all $\mathbf t\in \CT$,
\be\label{mvw}
  \operatorname n_{\mathbf
 t}(\pi)=\operatorname n_{-\mathbf
 t}(\pi^\vee).
\ee
Therefore the inequality \eqref{geqn} is implied by \eqref{geqn22}.

To prove the first assertion of the theorem, we first assume that $U$ is a real symplectic space, or a complex Hermitian or skew-Hermitian space. Then
there is a unique pair $(m_1, m_2)$ of integers so that
\[
   \left\{
     \begin{array}{l}
       m_1, m_2\in \{\dim \sigma\mid \sigma\in \widehat \CW_U, \,\sigma \textrm{ represents an element of $\CT$}\},  \\
       m_1+m_2=2\dim U+\mathrm d_{\rD, \epsilon}-2, \textrm{ and}\\
       m_1-m_2=0\textrm{ or } 2.
     \end{array}
   \right.
\]

As a first step, we show that there exists $\mathbf t_1\in \CT$ such
that $\operatorname n_{\mathbf t_1}(\pi)\leq m_1$. We pick any
$\mathbf t\in \CT$. If $\operatorname n_{\mathbf t}(\pi)\leq m_1$, we are done. Otherwise $\operatorname n_{\mathbf t}(\pi)\geq m_1+2\geq m_2+2$, and so $\pi\notin \CR_{\sigma_{\mathbf t, m_2}}$, where $\sigma_{\mathbf t, m_2}$ is the element of $\widehat \CW_U$ which represents $\mathbf t$ and has dimension
$m_2$. By applying
Proposition \ref{quotient4} to $\sigma_{\mathbf t, m_2}$, the same
proof as in Proposition \ref{dict0} shows that there exists an
element $\sigma'\in \widehat \CW_U$ such that
\[
 \left\{
   \begin{array}{l}
      \textrm{$\sigma'-\sigma_{\mathbf t, m_2}$ represents a non-zero element of $\CK_U$;}\smallskip \\
     \dim \sigma'=m_1; \textrm{ and}\\
     \pi\in \CR_{\sigma'}.
   \end{array}
 \right.
\]
Consequently, we have
\[
  \min\{\operatorname n_{\mathbf t'}(\pi)\mid \mathbf t'\in \CT, \,
  \mathbf t'\neq \mathbf t\}\leq  m_1.
\]
We may thus find some $\mathbf t_1\in \CT$ such that $\operatorname
n_{\mathbf t_1}(\pi)\leq m_1$.

Write $k=\operatorname n_{\mathbf t_1}(\pi)$, and consider $\sigma_{\mathbf t_1,
k-2}$, the element of $\widehat \CW_U$ which represents $\mathbf
t_1$ and has dimension $k-2$. Then $\pi\notin \CR_{\sigma_{\mathbf
t_1, k-2}}$, and
\[
 k-2\leq m_1-2\leq m_2\leq \frac{2\dim U+\mathrm d_{\rD, \epsilon}-2}{2}.
\]
Similarly, applying Proposition \ref{quotient4} to $\sigma_{\mathbf
t_1, k-2}$, the same  proof as in Proposition \ref{dict0} shows that
there exists an element $\sigma'\in \widehat \CW_U$ such that
\[
 \left\{
   \begin{array}{l}
      \textrm{$\sigma'-\sigma_{\mathbf t_1, k-2}$ represents a non-zero element of $\CK_U$;}\smallskip \\
     \dim \sigma'+(k-2) =2\dim U+\mathrm d_{\rD, \epsilon}-2; \textrm{ and}\\
     \pi\in \CR_{\sigma'}.
   \end{array}
 \right.
\]
Consequently, we have
\[
  \min\{\operatorname n_{\mathbf t'}(\pi)\mid \mathbf t'\in \CT, \,
  \mathbf t'\neq \mathbf t_1\}\leq  2 \dim U+\mathrm d_{\rD,\epsilon}-k.
\]
In other words, there is an element $\mathbf t_2\in \CT\setminus \{\mathbf t_1\}$ such that
\[
 \operatorname n_{\mathbf t_1}(\pi)+\operatorname n_{\mathbf
  t_2}(\pi)\leq 2\dim U+\mathrm d_{\rD,\epsilon}.
\]
In view of the inequality \eqref{geqn}, this proves the first assertion of the theorem, in the case when $U$ is a real symplectic space, or a complex Hermitian or skew-Hermitian space.

Now assume that $U$ is a real quaternionic skew-Hermitian space. Then $\CT=\CW_U$. Recall that for each $\mathbf t\in \CW_U$, $\dim \mathbf t\in \Z/2\Z$ denotes the parity of the dimension of an element of $\mathbf t$. Put
\[
 \operatorname n_{0}(\pi):=\min\{\operatorname n_{\mathbf t}(\pi)\mid \mathbf t\in\CW_U, \dim \mathbf t \textrm{ is even} \}
\]
and
\[
 \operatorname n_{1}(\pi):=\min\{\operatorname n_{\mathbf t}(\pi)\mid \mathbf t\in\CW_U, \dim \mathbf t \textrm{ is odd} \}.
\]
In view of the inequality \eqref{geqn}, for the  first assertion of
the theorem, it suffices to show that
\[
   \operatorname n_{0}(\pi)+ \operatorname n_{1}(\pi)\leq 2\dim U+1 \qquad (\mathrm d_{\rD,\epsilon}=1).
\]
Assume by contradiction that
\[
   \operatorname n_{0}(\pi)+ \operatorname n_{1}(\pi)\geq 2\dim U+3.
\]
Then there are integers $m$, $m'$ such that
\be\label{inemm}
 \left\{
   \begin{array}{l}
      \textrm{$m$ is even and $m'$ is odd;}\smallskip \\
       m+m'=2\dim U-1;\\
      m<\operatorname n_{0}(\pi);\,\, \textrm{ and}\\
  m'<\operatorname n_{1}(\pi);
    \end{array}
 \right.
\ee

Using Proposition \ref{quotient5}, the same proof as Proposition \ref{dict0} shows that there exists an element $\sigma\in \widehat \CW_U$ such that
\[
 \left\{
   \begin{array}{l}
      \dim \sigma=m \textrm{ or } m'; \textrm{ and}\\
     \pi\in \CR_{\sigma}.
   \end{array}
 \right.
\]
Therefore either $\operatorname n_{0}(\pi)\leq m$ or $\operatorname n_{1}(\pi)\leq m'$. This contradicts the two inequalities of \eqref{inemm}.

The last assertion of the Theorem is easily implied by
\eqref{consarch} and \eqref{geqn}.
\end{proof}

\vsp

\noindent {\bf Remarks}: (a) When $U$ is a complex symmetric bilinear space, the conservation relations were proved by Adams-Barbasch \cite{AB}
using the explicit duality correspondence. A. Paul proved the conservation relations for archimedean unitary-unitary dual pair correspondence \cite[Theorem 1.4]{Pa}, for a discrete series representation, or a representation
irreducibly induced from a discrete series representation.

(b)  Let $G$ be a topological group. An involutive continuous automorphism $\tau$ of $G$ is called an MVW-involution if $\tau(g)$ and $g^{-1}$ are conjugate in $G$, for all $g$ in an open dense subset of $G$.  MVW-involutions of $\oGb(U)$ do not exist in general when $U$  is a non-archimedean quaternionic Hermitian space or a non-archimedean quaternionic skew Hermitian space \cite[Proposition 1.3]{LST} (MVW-involutions of $\oGb(U)$ do exist in all the other cases). Nonetheless the equality \eqref{mvw} is still
valid for this case, in view of the equalities in \eqref{enequ}.
\vsp

\end{document}